\documentclass[reqno,oneside,12pt]{amsart}

\usepackage[T1]{fontenc}
\usepackage{times,mathptm}
\usepackage{amssymb,epsfig,verbatim,xypic,mathtools}
\usepackage{xcolor,url,hyperref}

\theoremstyle{plain}

\newtheorem{thm}{Theorem}[section]

\newtheorem{cor}[thm]{Corollary}
\newtheorem{pro}[thm]{Proposition}
\newtheorem{lem}[thm]{Lemma}
\newtheorem{proposition-principale}[thm]{Proposition principale}
\newtheorem{thm-principal}{Th\'eor\`eme principal}[section]

\newtheorem{thm-C}{Theorem}[section]

\theoremstyle{definition}

\newtheorem{eg}[thm]{Example}
\newtheorem{rem}[thm]{Remark}

\def\C{\mathbf{C}}
\def\bfk{\mathbf{k}}
\def\bfK{\mathbf{K}}
\def\bfF{{\mathbf{F}}}

\def\R{\mathbf{R}}
\def\Q{\mathbf{Q}}

\def\Z{\mathbf{Z}}
\def\N{\mathbf{N}}

\def\Id{{\mathsf{Id}}}

\def\bfe{\mathbf{e}}
\def\bfo{\mathbf{o}}
\def\bfw{\mathbf{w}}

\def\Hyp{{\mathbb{H}}}
\def\Hor{{\mathrm{Hor}}}
\def\ph{{\mathfrak{hor}}}
\def\Axis{{{\mathrm{Ax}}}}

\def\jj{{\sf{j}}}

\def\bbP{\mathbb{P}}
\def\bbA{\mathbb{A}}

\def\dist{\mathsf{dist}}

\def\Aut{{\sf{Aut}}}
\def\Bir{{\sf{Bir}}}

\def\PGL{{\sf{PGL}}\,}

\def\GL{{\sf{GL}}\,}
\def\Mat{{\sf{Mat}}\,}

\def\PSO{{\sf{PSO}}\,}

\def\SL{{\sf{SL}}\,}

\def\Isom{{\sf{Isom}}\,}

 \def\NS{{\mathrm{NS}}}
 \def\tr{{\mathrm{tr}}}
 \def\Tr{{\mathrm{Tr}}}
 \def\rk{{\mathrm{rank}}}

 \def\HP{{\mathfrak{p}}}

\def\Ind{{\mathrm{Ind}}}
\def\Exc{{\mathrm{Exc}}}

\def\Pic{{\mathrm{Pic}}}
\def\Irr{{\mathrm{Irr}}}
\def\Spec{{\mathrm{Spec\,}}}
\def\Frac{{\mathrm{Frac\,}}}
\def\Gal{{\mathrm{Gal\,}}}

\newcommand{\Ell}{\mathrm{Ell}}
\newcommand{\Jon}{\mathrm{Jon}}
\newcommand{\Hal}{\mathrm{Hal}}
\newcommand{\Lox}{\mathrm{Lox}}

\addtocounter{section}{0}             
\numberwithin{equation}{section}       

\begin{document}

\setlength{\baselineskip}{0.54cm}        
%
%
\title[Families]
{Three chapters on Cremona groups}
\date{2020}
\author{Serge Cantat}
\address{Univ Rennes, CNRS, IRMAR - UMR 6625, F-35000 Rennes, France}
\email{serge.cantat@univ-rennes1.fr}

\author{Julie D\'eserti}
\thanks{The second
 author was partially supported by the ANR grant Fatou ANR-17-CE40-
0002-01 and the ANR grant Foliage ANR-16-CE40-0008-01.}
\address{Universit\'e C\^ote d'Azur, Laboratoire J.-A. Dieudonn\'e, UMR 7351, Nice, France}
\email{Julie.Deserti@math.cnrs.fr}

\author{Junyi Xie}
\thanks{The third
 author was partially supported by the ANR grant Fatou ANR-17-CE40-
0002-01}
\address{Univ Rennes, CNRS, IRMAR - UMR 6625, F-35000 Rennes, France}
\email{junyi.xie@univ-rennes1.fr}

%
%

%
%

%
%


\maketitle

%
%

This article is made of three independent parts, the three of them concerning the Cremona group in $2$ variables, which can be read almost independently.
 It contains in particular the following results:

\medskip

{\bf{1.--}} In the first part, we answer a question of Igor Dolgachev, which is related to the following problem: {\sl{given 
a birational transformation $f\colon \bbP^m_\bfk\dasharrow \bbP^m_\bfk$, and a linear projective transformation $A\in \PGL_{m+1}(\bfk)$,
when is $A\circ f$ regularizable~?}} In other words, when do there exist a variety $X$ and a birational map $\varphi\colon X\dasharrow \bbP^m_\bfk$
such that $\varphi^{-1}\circ(A\circ f)\circ \varphi$ is regular ? Dolgachev's initial question  is whether this may happen for
all $A$ in $\PGL_{m+1}(\bfk)$, and the answer is negative in all dimensions $m\geq 2$ (Theorem~\ref{thm:dolgachev-answer}).

\medskip

{\bf{2.--}} In the second part, we look at the degree $\deg(f)$ of the formulas defi\-ning $f\in \Bir(\bbP^2_\bfk)$, and at the sequence $n\mapsto \deg(f^n)$. 
We show that there is no constraint on the oscillations of this function, that is on the sequence 
\begin{equation}
n\mapsto \deg(f^n)-\deg(f^{n-1})
\end{equation}
for {\sl{small}} values of $n$:
given any bounded interval $[0, N]$, the oscillations on that interval are arbitrary (see Theorem~\ref{thm:Oscillations}).

\medskip

{\bf{3.--}}  The third part studies the degree of pencils of curves which are invariant by a birational transformation.  In the most interesting case  (i.e. when $f$ is a Halphen or Jonqui\`eres 
twist),  we prove that this degree is bounded by a function of $\deg(f)$. 
We derive corollaries on the structure of conjugacy classes, and their properties with respect to the Zariski topology of $\Bir(\bbP^2_\bfk)$
(see Theorems~\ref{thm:bounds for halphen} and~\ref{thm:Halphen-Conjugacy-Classes} for instance, and \S~\ref{par:limits}).

\medskip

These three themes are related together by the properties of degrees of birational maps and their iterates and in particular to the
question:
how do the  sequence $(\deg(f^n))_{n\geq 0}$ and its asymptotic behavior vary when $f$ moves in a fa\-mily of birational transformations of the plane ?
A second feature of this article is that we address the problem of finding examples which are defined over number fields (i.e. over ${\overline{\Q}}$)
instead of $\C$ (see Theorems~\ref{thm:DolgachevQ} and~\ref{thm:realbounded} for instance).

\medskip

{\bf{Acknowledgements.--}} Thanks to Vincent Guirardel, St\'ephane Lamy, Anne Lonjou,  and Christian Urech for interesting discussions. We are also grateful to the referee for interesting questions and suggestions, from which we could improve greatly our article. 

\medskip

\begin{center}
{\bf{-- Part A. --}}
\end{center}

\subsection*{Introduction} Fix a degree $d\geq 2$, and denote by $\Bir_d(\bbP^m_\bfk)$ the variety
of birational maps of $\bbP^m_\bfk$ of degree $d$. An element $f$ of  $\Bir_d(\bbP^m_\bfk)$ is regularizable if there is 
a smooth variety $V$ and a birational map $\varphi\colon V\to \bbP^m_\bfk$ such that $\varphi^{-1}\circ f\circ \varphi$
is an automorphism of $V$; it is algebraically stable if $\deg(f^n)=\deg(f)^n$ for all $n\geq 0$. As we shall see, 
regularizability and algebraic stability are not compa\-tible: regularizable maps form a set of {\sl{special points}} in  
$\Bir_d(\bbP^m_\bfk)$, which does not intersect the set of algebraically stable maps. It would be interesting to classify 
the density loci of regularizable maps: for instance, what are the irreducible algebraic curves in $\Bir_d(\bbP^2_\bfk)$ containing an infinite set of regularizable maps ? Are regularizable maps everywhere dense in $\Bir_d(\bbP^2_\bfk)$ for 
the Zariski topology ? In this first part, we address this type of questions by looking at regularizable maps in 
orbits $A\circ f$, for $A\in \Aut(\bbP^m_\bfk)$. In particular, we prove (see Theorems~\ref{thm:dolgachev-answer} and~\ref{thm:DolgachevQ})

\medskip 

\noindent{\bf{Theorem.}} 
{\emph{Let $f$ be a birational transformation of $\bbP^m$ which is 
defined over a  field $\bfk$, {i.e.} the formulas defining $f$ have coefficients in $\bfk$, 
and assume that $\bfk$ is uncountable, or that its characteristic is zero.
Then, there exists an element $A$ of $\PGL_{m+1}(\bfk)$ such that $\deg((A\circ f)^n)=\deg(f)^n$ for all $n\geq 1$.}}
\section{Degrees and regularization}

\subsection{Regularization and degree growth}

Let $M$ be a smooth, irreducible projective variety of dimension $m$, defined over an algebraically closed field $\bfk$.
A birational transformation $f\colon M\dasharrow M$ is {\bf{regularizable}} if there exist a smooth projective variety
$V$ and a birational map $\varphi\colon V\dasharrow M$ such that 
\begin{equation}
f_V:= \varphi^{-1}\circ f\circ \varphi
\end{equation}
is an automorphism of $V$. 

Let $H$ be a polarization of $M$. Let $k$ be an integer between $0$ and $m$.
The degree of $f$ in codimension $k$ with respect to
$H$ is the positive integer 
given by the intersection product
\begin{equation}
\deg_{H,k}(f)=(f^*(H^k))\cdot (H^{m-k}).
\end{equation}
When $M$ is the projective space $\bbP^m_\bfk$, we shall choose the  polarization ${\mathcal{O}}(1)$ given by hyperplanes. 
For simplicity, we denote by $\deg_k(f)$ the degree when the polarization is clear or irrelevant, and  by 
$\deg(f)$ the degree in codimension $1$. 

Say that the sequence $(u_n)$ dominates $(v_n)$ if $\vert v_n\vert \leq A\vert u_n\vert + B$ for some constants $A$ and $B$, 
and that $(u_n)$ and $(v_n)$ have the same growth type if each one dominates the other: this defines an equivalence relation on 
sequences of real numbers. Then, the growth type of 
the sequence $(\deg_k(f^n))_{n\geq 0}$ is a conjugacy invariant and does not depend on the 
polarization: $(\deg_{H,k}(f^n))$ and $(\deg_{H',k}(f_V^n))$ are equivalent sequences  if $f_V$ 
is conjugate to $f$ via a birational map $\varphi \colon V\dasharrow M$ and $H'$ is a polarization of $V$. In particular, the $k$-th dynamical degree 
\begin{equation}
\lambda_k(f)=\lim_{n\to +\infty} \deg_{H,k}(f^n)^{1/n}
\end{equation}
is  invariant under conjugacy and  
change of polarization (see \cite{NguyenBD:2017, TTTruong}). 

\subsubsection{Bounded degrees} 
The degree sequence $(\deg_{H,k}(f^n))_{n\geq 0}$ is bounded for at least 
one codimension $1\leq k\leq m-1$ if and only if it is bounded for all of them; 
in that case, a theorem of Weil asserts that $f$ is regularizable (\cite{Weil:1955, Edixhoven-Romagny:Weil, Kraft:regularization, Zaitsev:1995}).
For instance, automorphisms of $\bbP^{m_1}\times \cdots \times \bbP^{m_k}$ have bounded degree sequences.

\subsubsection{Exponential degree growth} 
There are also birational transformations of the projective space $\bbP^m_\bfk$ which are regularizable and have an exponential 
degree growth (see \cite{Diller-Favre} for instance). Such a transformation is not conjugate to an 
automorphism of $\bbP^m_\bfk$ because the growth type is a conjugacy invariant. 

\begin{pro}\label{pro:exp-degree-growth}
Let $f$ be a birational transformation of a smooth, irreducible, projective variety $M$. 
If, for some $1\leq k\leq \dim(M)-1$, 
the dynamical degree $\lambda_k(f)$ is an integer $>1$, 
then $f$ is not regularizable. 
\end{pro}

\begin{proof}
If $f_V$ is an automorphism of a smooth projective variety which is conjugate to $f$ via a birational map $\varphi\colon V\dasharrow M$, 
then $\lambda_k(f)=\lambda_k(f_V)$  is the largest eigenvalue of $(f_V)^*$ on the space $N^k(X)$ of numerical 
classes of codimension $k$ cycles. This   follows from the invariance of $\lambda_k$ under conjugacy, the existence of a salient, convex
and $(f_V)^*$-invariant cone in  $N^k(V;\R)$, and the Perron-Frobenius theorem (see \cite{NguyenBD:2017}). 
The abelian group $N^k(V;\Z)$ determines an $(f_V)^*$-invariant
lattice in $N^k(V;\R)$. On the other hand, if $g$ is an isomorphism of a lattice 
$\Z^\ell$ and $\lambda$ is an eigenvalue of $g$, then either $\lambda=\pm 1$ or
$\lambda$ is not an integer; indeed, if $\lambda\in \Z$, the eigenspace 
corresponding to the eigenvalue $\lambda$ is defined over $\Q$, hence it intersects
$\Z^\ell$ on a submodule $L\simeq \Z^r$ with $1\leq r\leq \ell$: since $g_{\vert L}\colon L\to L$ 
is at the same time an invertible endomorphism and the multiplication by $\lambda$, we obtain
$\lambda^r=\det(g_{\vert L})=\pm 1$. These two remarks conclude the proof. 
\end{proof}

As a corollary, \textsl{if $f$ is a birational transformation of $\bbP^m_\bfk$ with (i) $\deg(f)>1$ and (ii) $\deg(f^n)=\deg(f)^n$
for all $n\geq 1$, then $f$ is not regularizable}. 

\subsubsection{Polynomial degree growth}

Another example 
is given by Halphen and Jonqui\`eres twists of surfaces. If $X$ is a surface, and $f$ is a birational transformation of $X$ such that $\deg(f^n)\simeq \alpha n^2$
for some constant $\alpha >0$ ({i.e.} $f$ is a Halphen twist), then $f$ is regularizable on a surface $Y\dasharrow X$ which supports an $f_Y$-invariant genus $1$
fibration (see \cite{Diller-Favre, Gizatullin:1980, Cantat-Dolgachev}, and Section~\ref{par:Halphen_Jonquieres_basics} below). If $f$ is a birational transformation of a surface $X$, and $\deg(f^n)\simeq \alpha n$
for some constant $\alpha >0$ ({i.e.} $f$ is a Jonqui\`eres twist), then $f$ preserves a unique pencil of rational curves, but $f$ is not regula\-rizable (see \cite{Diller-Favre, Blanc-Deserti:2015}, and Section~\ref{par:Halphen_Jonquieres_basics} below). 
Here is a  statement that generalizes this last remark.

\begin{pro}
Let $f$ be a birational transformation of a variety $M$ with 
$$
\limsup_{n\to +\infty}\frac{1}{n^\ell}\deg_k(f^n)=\alpha$$ for some $1\leq k\leq \dim(M)$, some $\alpha \in \R_+^*$ and some {\sl{odd}} integer~$\ell$. 
Then, $f$ is not regularizable.
\end{pro}

Before proving this result, note that there are examples of such birational transformations in all
dimensions $m \geq 2$ (see \cite{Diller-Favre, Deserti}).

\begin{proof}
  If $f$ were regularizable, there would be an automorphism $f_V$
of a projective variety $V$ whose action on $N^k(V)$ would satisfy $\limsup n^{-\ell}\parallel (f_V^*)^n\parallel= \alpha$, 
where $\parallel \cdot\parallel$ is any norm on ${\mathsf{End}}_\R(N^k(V))$. 
On the other hand $f_V^*$ preserves the integral structure of $N^k(V)$; it is given by an element of 
$\GL_r(\Z)$ if $r$ is the rank of $N^k(V)$. So, all eigenvalues of $f_V^*$ are algebraic integers, and since the 
growth property implies that the spectral radius of $f_V^*$ is $1$, we deduce from Kronecker's theorem 
that its eigenvalues are roots of $1$.  Thus some iterate $(f_V^*)^s$ is
a non-trivial unipotent matrix. Now, the growth property shows that the largest Jordan block of $(f_V^*)^s$ has size $\ell+1$. 
But $f_V^*$ and its inverse $(f_V^{-1})^*$ preserve a salient, open, convex cone
(see~\cite{NguyenBD:2017}), and the
existence of such an invariant cone is not compatible with a maximal Jordan block of even size, because 
if one takes a generic vector $u$ in the interior of such a cone, then $\lim_{+\infty} \vert n\vert^{-\ell} (f_V^*)^n(u)$ and $
\lim_{-\infty} \vert n\vert ^{-\ell} (f_V^*)^n(u)$ would be non-zero opposite vectors of that cone (\cite{Marquis:2014}, Proposition 2.2).
\end{proof}

\subsection{Semi-continuity and dimension $2$}

Let $\Bir_d(\bbP^m_\bfk)$ be the algebraic variety of all birational transformations of $\bbP^m_\bfk$ of degree $d$ (with respect to the polarization ${\mathcal{O}}(1)$). When $d$ is large these va\-rieties have many distinct components, of various dimensions. This is shown in \cite{BisiCalabriMella}
when $m=2$; in dimension $m\geq 3$ it is better to look at the subvarieties obtained by fixing the $m-1$ degrees
$\deg_1(f)$, $\ldots$, $\deg_{m-1}(f)$, but even these varieties may have many components.

The group $\Aut(\bbP^m_\bfk)=\PGL_{m+1}(\bfk)$ acts by left translations, by right translations, and by conjugacy on 
$\Bir_d(\bbP^m_\bfk)$. Since this group is connected, these actions preserve each connected component. 

\begin{thm}[Junyi Xie, \cite{Xie:Duke}]\label{thm:Xie} Let $\bfk$ be a field, and let $d$ be a positive integer.
\begin{enumerate} 
\item In $\Bir_d(\bbP^2_\bfk)$ the subsets 
$
Z(d;\lambda) =\left\{ f\in \Bir_d(\bbP^2)\; \vert \; \; \lambda_1(f)\leq \lambda \right\}
$
are Zariski closed. 

\item For every $f\in \Bir_d(\bbP^2_\bfk)$, the set 
$
B(f;\lambda)= \left\{ A\in \PGL_3(\bfk)\; \vert \; \lambda_1(A\circ f)\leq \lambda\right\}
$
is a Zariski closed subset of $\PGL_3(\bfk)$. 

\item When $\lambda<d$, $Z(d;\lambda)$ is a nowhere dense subset of $\Bir_d(\bbP^2_\bfk)$
and $B(f;\lambda)$ is a proper subset of $\PGL_3(\bfk)$ $($for every  $d>1$ and $f\in \Bir_d(\bbP^2_\bfk))$. 

\end{enumerate}
\end{thm}

Note that the first assertion says that $\lambda_1$ is a lower-semicontinuous function for the Zariski topology.

\begin{proof} The first statement follows from the semi-continuity result obtained in \cite{Xie:Duke}, Theorem~1.5. 
The second is a direct consequence of the first, because the action of $\PGL_3(\bfk)$ on $\Bir_d(\bbP^2_\bfk)$ is algebraic. 
The third assertion is a consequence of Theorem~1.5 below and of its proof (see also the pages 917-918 of \cite{Xie:Duke}). 
\end{proof}

Fix $d\geq 2$, and list the 
dynamical degrees  
$\delta_1(d)<\delta_2(d)< \ldots < \delta_j(d)\ldots \leq d$ which are realized by birational maps of the plane of degree
$d$; this is, indeed,  the union of $\{ d\}$ and of an increasing sequence $(\delta_i(d))$(see \cite{Urech, Blanc-Cantat}). Choose an element $g$ of $\Bir_d(\bbP^2_\bfk)$. The Zariski closed subsets $B(g;\delta_i(d))$ form a countable family 
of proper Zariski closed subsets of $\PGL_3(\bfk)$. Thus, Theorem~\ref{thm:Xie} and Proposition~\ref{pro:exp-degree-growth} provide the following result.

\begin{cor}\label{coro:Xie}
Assume that $\bfk$ is not countable. If $f$ is a birational map of $\bbP^2_\bfk$ of degree $d\geq 2$, then there
exists an automorphism $A$ of the plane such that $\lambda_1(A\circ f)=d$  $($and $A\circ f$ is $1$-stable, see~\cite{Diller-Favre}$)$. In particular, $A\circ f$ is not 
regularizable.
\end{cor}

\subsection{Dolgachev's question} This is related to the following question, asked by Dolgachev to the second author of this paper  for $m=2$: 

\smallskip

\noindent{\bf{Dolgachev's question.--}} {\sl{Set
\begin{equation}
Reg(f):=\left\{ A\in \Aut(\bbP^m_\bfk)\; \vert \; A\circ f \; {\text{ is regularizable}} \right\}.
\end{equation}
Does there exist a birational transformation $f$ of $\bbP^m_\bfk$ 
such that $\deg(f)>1$ and $Reg(f)$ is equal to $\Aut(\bbP^m_\bfk)$ ?}} 

\smallskip

This question depends on $m$ and $\bfk$. Corollary~\ref{coro:Xie} provides a negative answer for $m=2$, at least if $\bfk$ is uncountable. 
We progressively extend the range of application of this argument in higher dimension: in \S~\ref{par:Dolgachev-Higher-Dimension} we consider any uncountable field $\bfk$, and in \S~\ref{par:Dolgachev-Qbar} we consider any field $\bfk$ of characteristic zero.

\subsection{Higher dimensions over uncountable fields}\label{par:Dolgachev-Higher-Dimension}


If $f$ is a birational transformation of $\bbP^m_\bfk$, we denote by $\Ind(f)$ its indeterminacy set, and 
by $\Exc(f)$ its exceptional locus: it is the union of the Zariski closures
of all proper
Zariski closed subsets $W$ of $\bbP^m_\bfk\setminus\Ind(f)$ whose strict transform $f_\circ(W)$ satisfies $\dim(f_\circ(W))<\dim(W)$.
The exceptional locus is a non-empty union of hypersurfaces, except when $f$ is in $\Aut(\bbP^m_\bfk)$.

Consider two birational transformations $f$ and $g$ of $\bbP^m_\bfk$. Then, $\deg(f\circ g)\leq \deg(f)\deg(g)$
and this inequality is strict if and only if one can find a hypersurface $W\subset \bbP^m_\bfk$ which is contracted
by $f$ into the indeterminacy locus $\Ind(g)$. Thus, $\deg(f^n)=\deg(f)^n$, unless there is an 
hypersurface in the exceptional locus $\Exc(f)$ which is ultimately mapped in $\Ind(f)$ by some positive
iterate $f^n$.

\begin{thm}\label{thm:dolgachev-answer}
Let $f$ be a birational transformation of $\bbP^m_\bfk$ of degree $d\geq 2$. The set of automorphisms
$A$ in $\PGL_{m+1}(\bfk)$ such that $\deg((A\circ f)^n)\neq (\deg(A\circ f))^n$ for some $n>0$ is
a countable union of proper Zariski closed subsets of $\PGL_{m+1}(\bfk)$.
\end{thm}

To prove this result, we present a simple variation on Xie's argument for the third assertion of Theorem~\ref{thm:Xie}.

\begin{proof}
Let $D_1$, $\ldots$, $D_\ell$ be the irreducible components of codimension $1$ in $\Exc(f)$. Let $C_{j,0}=f_\circ D_j$ be the
strict transform of $D_j$. The subsets 
$C_{j,0}$ are Zariski closed, of codimension $\geq 2$.  For $A$ in $\PGL_{m+1}(\bfk)$, set $g_A=A\circ f$. Then, 
the exceptional and indeterminacy loci satisfy $\Exc(g_A)=\Exc(f)$, $\Exc(g_A^{-1})=A(\Exc(f^{-1}))$, 
$\Ind(g_A)=\Ind(f)$, and $\Ind(g_A^{-1})=A(\Ind(f^{-1}))$.
For every index $1\leq j\leq \ell$, and every integer
$n\geq 0$, consider the subset $F(j;n)\subset  \PGL_{m+1}(\bfk)$ of automorphisms $A$ such that 
$A(C_{j,0})$ is not contained in $\Ind(f)$, and then $g_A(A(C_{j,0}))$ is not contained in $\Ind(f)$, 
up to $g_A^n(A(C_{j,0}))$ which is not contained in $\Ind(f)$ either. 
This set $F(j,n)$ is Zariski open in $\PGL_{m+1}(\bfk)$, and to prove the theorem our goal is to show that all these Zariski 
open subsets are non-empty. This follows from Lemma~\ref{lem:transversality} below. 
\end{proof}

\begin{lem}\label{lem:transversality}
Let $W$ and $Q$ be non-empty, irreducible, Zariski closed subsets of $\bbP^m_\bfk$, 
with $\Ind(f)\subset W\neq \bbP^m_\bfk$.
Let $n\geq 0$ be an integer. There exists an element $A$ of $ \PGL_{m+1}(\bfk)$ such that 
$A(Q)$ is not contained in $W$, $(A\circ f)(A(Q))$ is not contained in $W$, $\ldots$, $(A\circ f)^n(A(Q))$ is not contained in $W$.
\end{lem}

\begin{proof}
Pick a nonzero homogeneous polynomial $P$ which vanishes on $W$. After replacing $W$ by an irreducible component of $\{P=0\}$, we  assume that $W$ is a hypersurface.
Start with $n=0$. We look for $A$ such that $A(Q)$ is not contained in $W$. The set of such matrices $A$ is Zariski open and
non-empty, because $W$ has codimension $1$ and the action of  $\PGL_{m+1}(\bfk)$ is transitive. 

Now, assume that the property is not satisfied for $n=1$. This means that among the elements $A$ in $\PGL_{m+1}(\bfk)$ 
such that $A(Q)$ is not contained in $W$, all of them satisfy $(A\circ f)(A(Q))\subset W$; fixing an equation $P$ of $W$, we obtain 
\begin{equation}\label{eq:incidence}
P((A\circ f)A(z))=0 \quad (\forall z\in Q \setminus A^{-1}\Ind(f)).
\end{equation}
This equation makes sense on the set of points $(z,A)$ in $Q\times \GL_{m+1}(\bfk)$ such that $A(z)\notin \Ind(f)$ and, 
more generally, on the set of points $(z,A)$ in $Q\times \Mat_{m+1}(\bfk)$ such that $z$ is not in the kernel of the linear
map $A$ and $A(z)$ is not in $\Ind(f)$. Now, take such a pair $(z,A)$ with $A$ of rank one. We choose $A$ such that
\begin{itemize}
\item[(1)] the image  of $A$  
projects to a point $q_A$ of $\bbP^m_\bfk$  which is not contained in $W$ (in 
particular, $q_A$ is not in $\Ind(f)$);
\item[(2)] the projection  $K_A\subset \bbP^m_\bfk$ of  its kernel does not contain  $f(q_A)$
and does not contain $Q$. 
\end{itemize}
Then, for every point $z\in Q\setminus K_A$ we obtain $(A\circ f)(A(z))=A(f(q_A))=q_A\notin W$ and $P(A \circ f)(A(z))\not=0$. This provides a contradiction 
with Equation~\eqref{eq:incidence}. 

Thus, the set of automorphisms $A$ for which $A(Q)$  and $(A\circ f)(A(Q))$ are not contained in $W$ is generic in $\PGL_{m+1}(\bfk)$. 
The same argument proves the lemma by recursion (choosing the same matrix of rank one).
\end{proof}

\begin{cor}\label{cor:dolgachev-higher-dimension}
Let $\bfk$ be an uncountable, algebraically closed field. Let $f$ be a birational transformation 
of  $\bbP^m_\bfk$ of degree $d\geq 2$. There exists an automorphism $A$ in $\PGL_{m+1}(\bfk)$
such that $A\circ f$ is not regularizable. 
\end{cor}

\begin{rem}
Let $f$ be a birational transformation of a projective variety $M$. Assume that $f$ contracts an irreducible 
subvariety $D_0$ of $V$ onto a Zariski closed subset $C_0$ (i.e. $C_0=f_\circ (D_0)$ and $\dim(C_0)<\dim(D_0)$). 
Then, assume that the positive orbit of $C_0$, namely $C_1=f_\circ (C_0)$, $\ldots$, $C_{n+1}=f_\circ (C_n)$ 
is made of subsets which are never contained in $\Ind(f)$. Suppose, in the opposite time direction, that the preimages
$D_{-n-1}=f^\circ(D_{-n})$ are never contracted. Then, the $C_n$ and the $D_{-n}$ form two sequences of pairwise 
distinct algebraic subsets; and if the union of the $D_{-n}$ is Zariski dense in $M$ (this is automatic if they have
codimension $1$), the transformation $f$ is not regularizable. Indeed, if $\varphi\colon V\dasharrow M$ is 
a birational map, one can find large integers $N>1$ such that $D_{-N}$ is not contracted  and $C_N$ is 
not blown-up by  $\varphi^{-1}$. Then, $f_V=\varphi\circ f\circ \varphi^{-1}$ cannot be an automorphism because
$f_V^{2N}$ contracts $\varphi^{-1}(D_{-N})$ onto $\varphi^{-1}(C_N)$.

This argument provides the following extension of Corollary~\ref{cor:dolgachev-higher-dimension}: {\sl{
Let $\bfk$ be an uncountable, algebraically closed field. Let $f$ be a birational 
transformation of a homogeneous variety $M=G/H$. Assume that $\Exc(f)$ is not empty. Then there is
an element $A$ of $G$ such that $A\circ f$ is not regularizable.}}
\end{rem}

\subsection{Dolgachev's question in characteristic zero}\label{par:Dolgachev-Qbar}

\begin{thm}\label{thm:DolgachevQ} Let $\bfk$ be a field of characteristic zero.
Let $f$ be a birational transformation of $\bbP^m$ which is 
defined over the field $\bfk$, i.e. the formulas defining $f$ have coefficients in $\bfk$. 
Then, there exists an element $A$ of $\PGL_{m+1}(\bfk)$ such that $\deg((A\circ f)^n)=\deg(f)^n$ for all $n\geq 1$.
\end{thm}

Theorem~\ref{thm:dolgachev-answer} does not imply directly this result: a countable union of proper Zariski closed 
subsets of $\PGL_{m+1}(\C)$ can contain $\PGL_{m+1}(\bfk)$ if $\bfk$ is countable.

\begin{proof}
Write the formulas for $f$ and $f^{-1}$ in homogeneous coordinates, and list the coefficients of these formulas: $c_1$, $\ldots$, $c_N\in \bfk$. They generate a ring $R=\Z[c_1, \ldots, c_N]$, and we can find a number field $K$ and a homomorphism of rings $\iota\colon R\to K$ such that 
\begin{enumerate}
\item the fraction field of $\iota(R)$ coincides with $K$: $\Frac(\iota(R))=K$; 
\item when $\iota$ is applied to the coefficients of the formulas defining $f$ and $f^{-1}$, one gets birational 
maps $f_\iota$ and $f_\iota^{-1}$ such that $\deg(f)=\deg(f_\iota)$, $\deg(f^{-1}) = \deg(f_\iota^{-1})$, 
and $f_\iota\circ f_\iota^{-1}=\Id$.
\end{enumerate} 
If we prove the result for $f_\iota$, over the number field $K$ in place of $\bfk$, one gets 
a matrix $A_\iota\in \GL_{m+1}(K)$ such that $\deg(A_\iota\circ f_\iota)^n=\deg(f)^n$. The coefficients $a_{ij}$ of $A_\iota$
being in $K=\Frac(\iota(R))$, there is a matrix $A=[a_{ij}]$ with coefficients in $\Frac(R)$ such that $\iota(A):=[\iota(a_{ij})]$
is equal to $A_\iota$. Then, $A\circ f$ satisfies the desired conclusion. 
So, from now on, we assume that $\bfk$ is a number field, and we denote by ${\mathcal{O}}_\bfk$ its ring of integers.
We are going to use the theorem of Chebotarev and the weak approximation theorem for such rings of integers
(see \cite[Theorems 1.2 and 7.7]{Platonov-Rapinchuk:Book}).

If $\deg(f)=1$, we can just take $A=\Id$, so we assume $\deg(f)\geq 2$. 

The exceptional locus of $\Exc(f)$ is a non-empty hypersurface, defined by an equation $J(x_0, \ldots, x_m)=0$
with coefficients in ${\mathcal{O}}_\bfk$; let $J=\prod_{i=1}^sJ_i$ be the decomposition of $J$ into irreducible factors over $\bfk$, i.e. 
as elements of $\bfk[x_0, \ldots, x_m]$, and let
$(W_i)_{i=1}^s$ be the corresponding irreducible components of $\Exc(f)$. We are loo\-king for an element $A$ of $\PGL_{m+1}(\bfk)$
such that $A\circ f$ and its iterates $(A\circ f)^n$ do not map $W_i$ into $\Ind(f)$, for all $i=1, \ldots s$ and all $n\geq 1$. Note that the 
$W_i$ may be reducible over ${\overline{\bfk}}$; but if the strict transform of $W_i$ by $(A\circ f)^n$ is not contained in $\Ind(f)$, 
then none of its irreducible components is mapped into $\Ind(f)$ (because $A\circ f$, $\Ind(f)$, and $W_i$ are all defined over $\bfk$).

Let $D_i$ be the intersection of 
$\Ind(f)$ with $W_i$: it is a Zariski closed subset of positive codimension in $W_i$. There exists a line $L_i$, 
defined over $\bfk$, such that $(L_i\cap W_i)(\overline{\bfk})$ is non-empty and $L_i$ does not intersect 
$\Ind(f)$(\footnote{\samepage Take a point $a\in \bbP^m(\bfk)$ outside $W_i\cup \Ind(f)$, and consider the projective space
$\bbP^{m-1}$ of lines through $a$; the linear projection that maps a point $b\neq a$ to the line $(ab)$ provides
a ramified cover $\pi\colon W_i\to \bbP^{m-1}$. The codimension of $\pi(\Ind(f))$ is positive because $\dim(\Ind(f))\leq m-2$, and this implies
that there are points of $\bbP^{m-1}(\bfk)$ in the complement of $\pi(\Ind(f))$. This corresponds to a line $L_i$
with the desired properties.}); in particular, $L_i$ does not intersect $D_i$. Restricting $J_i$ to the line $L_i$, we get a homogeneous polynomial in one
variable, with integer coefficients, that we denote $P_i$. By the theorem of Chebotarev, there are infinitely many 
maximal ideals ${\mathfrak{p}}$ of ${\mathcal{O}}_\bfk$  such that
$P_i$ splits into linear factors in the finite field ${\mathcal{O}}_\bfk/{\mathfrak{p}}$ (see~\cite{Stevenhagen-Lenstra} and~\cite{Janusz:ANF}); set $\bfF ={\mathcal{O}}_\bfk/{\mathfrak{p}}$.

We shall now work modulo such an ideal.
Write $f$ and its inverse $f^{-1}$ in homogeneous coordinates, with homogeneous formulas
whose coefficients are in ${\mathcal{O}}_\bfk$. Denote
by $\overline{f}$ the reduction of (the formulas defining) $f$ modulo ${\mathfrak{p}}$; we choose
${\mathfrak{p}}$ large enough to be sure that ${\overline{f}}$ is a birational map, of degree equal to the degree of $f$, 
and with inverse equal to ${\overline{f^{-1}}}$. We denote by ${\overline{L_i}}$ the reduction of $L_i$ modulo ${\mathfrak{p}}$, etc.
For ${\mathfrak{p}}$ large, the line ${\overline{L_i}}$ does not intersect $\Ind({\overline{f}})$. 
So, fix such an ideal, and pick an $\bfF$-point $x$ of ${\overline{L_i}}\cap {\overline{W_i}}$, corresponding to one of the roots of $\overline{P_i}$. 
This point is not in $\Ind(f)$, 
hence $\overline{f}$ is regular on an open neighborhood ${\overline{U}}$ of $x$; set ${\overline{U}}_i={\overline{U}}\cap {\overline{W_i}}$.
Since $\SL_{m+1}(\bfF)$ 
acts transitively on $\bbP^m(\bfF)$ there is an element $A_i\in \SL_{m+1}(\bfF)$ such that $A_i(\overline{f}(x))=x$.
Then $A_i\circ \overline{f}$ satisfies the following property: for every integer $n\geq 0$, $(A_i\circ {\overline{f}})^n$ is regular at $x$ and 
$(A_i\circ {\overline{f}})^n(x)=x$. This implies that {\sl{the strict transform of $\overline{W_i}$ by $(A_i\circ {\overline{f}})^n$ is not contained in $\Ind(f)$}}.

We do such a construction for each of the $s$ irreducible components $W_i$ of $\Exc(f)$. 
This gives a finite sequence of ideals ${\mathfrak{p}}_i$ 
and matrices $A_i\in \SL_{m+1}(\bfF_{i})$, with $\bfF_i={\mathcal{O}}_\bfk/{\mathfrak{p}}_i$; since for each $i$ we have infinitely many possible choices for ${\mathfrak{p}}_i$, we can choose them pairwise distinct. By the weak approximation theorem, there is a matrix $A\in \SL_{m+1}({\mathcal{O}}_\bfk)$ such that 
$A\equiv A_i\mod({\mathfrak{p}}_i)$ for every $i=1, \ldots, s$. Set $g_A=A\circ f$, and pick an irreducible component $W_i$ of $\Exc(f)$. Then,  $(g_A)^n$
never maps $W_i$ into $\Ind(g_A)=\Ind(f)$, because otherwise its reduction modulo ${\mathfrak{p}}_i$ would map ${\overline{W_i}}$ into $\Ind({\overline{f}})$. 
This shows that no component of $\Exc(g_A)$ is mapped into $\Ind(g_A)$ and, according to~\cite[\S~1.4]{Sibony:Panorama} this implies $\deg(g_A^n)=\deg(g_A)^n=\deg(f)^n$ for all $n\geq 0$.
\end{proof}

\medskip

\begin{center}
{\bf{-- Part B. --}}
\end{center}

\subsection*{Introduction} Possible sequences of degrees $(\deg(f^n))$, say for 
rational transformations of $\bbP^m$ and  for the polarization 
${\mathcal{O}}(1)$, are hard to describe (see~\cite{Deserti, Urech}). For instance, 
let $(u_n)$ be a sequence of positive integers that satisfy a linear recurrence
relation. Does there exist a dimension $d$ and a monomial map $f_A$
of $\bbP^d_\C$, given by some $d\times d$ matrix with integer coefficients, 
such that $\deg(f_A^n)=u_n$ for every $n\geq 1$. A similar question can be
asked for endomorphisms of abelian varieties instead of monomial maps. 
Fixing the dimension, one expects important constraints on such sequences.
For instance, there are strong constraints on $(\deg(f^n))_{n\geq 0}$ when $f\in
\Bir(\bbP^2_\bfk)$: 
\begin{itemize}
\item[(a)] if $\deg(f^2) > 3^{19} \deg(f)$ 
then $\lambda_1(f)>1$ (see~\cite{Xie:Duke}), and if $\lambda_1(f)>1$ then it is a Pisot or Salem number, larger than the Lehmer number (see~\cite{Diller-Favre, Blanc-Cantat});
\item[(b)] if $\deg(f^n)_{n\geq 0}$ does not grow exponentially fast, then asymptotically we have $\deg(f^n)\simeq \alpha(f) n^2$ or $\alpha(f) n$ for some $\alpha(f)\in \Q_+^*$. If the growth is linear, $\alpha=m/2$ for some 
integer $m\leq \vert \Ind(f)\vert \leq 3\deg(f)-3$ (see \cite{Blanc-Deserti:2015}).
\end{itemize}
 
In higher dimension, we refer to \cite{Dang-Favre} for a recent breakthrough. Here, we consider the case of bounded degree sequences, and show that, already in dimension $2$, there is no constraint on their initial oscillation:
\medskip 

\noindent{\bf{Theorem.}}  (see Theorems~\ref{thm:Oscillations} and~\ref{thm:realbounded}) 
{\emph{Let $D\colon \N^*\to \N$ be a function which vanishes identically on the complement of a finite set. 
There exist a birational transformation $h\colon \bbP^2_{\overline{\Q}}\dasharrow \bbP^2_{\overline{\Q}}$ and an integer $d\geq 1$ such that $\deg(h^n)=d-D(n) \quad {\text{for all }} \;  n\geq 1.$
}}

\medskip 

Here,  $\N$ denotes the set of non-negative integers, and $\N^*$ the set of positive integers.
This theorem shows that there is no constraint on the initial additive oscillations $\deg(f^{n+1})-\deg(f^n)$, for 
$n\geq 1$, even if we look for birational transformations defined over $\overline\Q$. The above Assertion (a) 
gives a strong constraint on the initial multiplicative oscillations $\deg(f^{n+1})/\deg(f^n)$.

This theorem is obtained in two steps. First, we prove it  
over an uncountable field instead of $\overline\Q$ (see Theorem~\ref{thm:Oscillations}). Then,  we show that 
bounded degree sequences   realized by elements of $\Bir(\bbP^m_\C)$ are also 
realized by elements of $\Bir(\bbP^2_{\overline\Q})$; we suspect that such a result holds for
all degree sequences: see \S~\ref{par:bds}.


\section{Oscillations}
 
\begin{thm}\label{thm:Oscillations}
Let $D\colon \N^*\to \N$ be a function which vanishes identically on the complement of a finite set. 
Let $\bfk$ be an uncountable, algebraically closed field.
There exist a birational transformation $h\colon \bbP^2_\bfk\dasharrow \bbP^2_\bfk$ and an integer $d\geq 1$ 
such that 
$
\deg(h^n)=d-D(n) \quad {\text{for all }} \;  n\geq 1. 
$
When $\bfk$ has characteristic $0$, one may find $h$ in $\Bir(\bbP^2_{\overline{\Q}})$.
\end{thm}

This means that $\deg(h^n)$ oscillates arbitrarily around its asymptotic value $d$, at least on the 
support of $D$. Since
$\deg(h^n)=d$ for large values of $n$, $h$ is an elliptic transformation; in fact, we shall construct 
$h$ as a conjugate $g\circ A\circ g^{-1}$ of a linear projective transformation $A\in \Aut(\bbP^2_\bfk)$
by a Jonqui\`eres transformation~$g$. 

The first part of this theorem is proved in Sections~\ref{par:prelim-oscill} to \ref{par:Jonq-conclusion}; the last assertion is 
a corollary of Theorem~\ref{thm:realbounded} from Section~\ref{par:realization-qbar}. 

\subsection{Preliminaries (see~\cite{Cantat:Annals, Cantat:Survey, Favre:Bourbaki, Manin:CubicForms})}\label{par:prelim-oscill}
Recall that the Picard-Manin space is generated by the class $\bfe_0$ of the polarization ${\mathcal{O}}(1)$
(i.e. the class of a line in $\bbP^2$) and by the classes $\bfe(p)$ of the exceptional divisors $E_p$ which 
are obtained by blowing up proper or infinitely near points $p$ of the plane.  

Let $A$ be an element of $\Aut(\bbP^2_\bfk)=\PGL_3(\bfk)$. It acts by automorphisms on $\bbP^2(\bfk)$; for every 
point $p\in \bbP^2(\bfk)$, $A$ lifts to an isomorphism from the blow-up of $\bbP^2$ at $p$ to the blow-up at $A(p)$,
sending the exceptional divisor $E_p$ to $E_{A(p)}$. Thus, its action on the Picard-Manin space satisfies 
\begin{equation}
A_\bullet\bfe_0=\bfe_0\quad {\text{ and }}\quad A_\bullet \bfe(p)=\bfe(A(p))
\end{equation}
for every proper or infinitely near point $p$ of $\bbP^2(\bfk)$ (see \cite{Cantat:Survey}). 

Let $g$ be a birational transformation of the plane, and let $h$ be the birational transformation 
$h=g\circ A \circ g^{-1}.$
The sequence of degrees $(\deg(h^n))_n$ is given by
\begin{equation}
\deg(h^n)= \langle (g\circ A^n \circ g^{-1})_\bullet \bfe_0 \vert \bfe_0 \rangle = \langle  A^n_\bullet( (g^{-1})_\bullet \bfe_0) \vert g^{-1}_\bullet \bfe_0 \rangle. 
\end{equation}
Write 
\begin{equation}
g^{-1}_\bullet \bfe_0=\deg(g)\bfe_0 - \sum_{i=1}^\ell a_i \bfe(p_i)
\end{equation}
for a finite set of base points $p_i$ and multiplicities $a_i>0$;
we assume for simplicity that there is no infinitely near point involved in this formula:
\begin{itemize}
\item[(Hyp. 1)] Each $\bfe(p_i)$ is the class of a proper point $p_i\in \bbP^2(\bfk)$. 
\end{itemize}
In particular, the $p_i$ are exactly the indeterminacy points of $g$, and the $a_i$ are their multiplicities. 
Then, $( A^n \circ g^{-1})_\bullet \bfe_0= \deg(g)\bfe_0 - \sum_{i=1}^\ell a_i \bfe(A(p_i))$ and we obtain
\begin{equation}
\deg(h^n)=\deg(g)^2-\sum_{i,j=1}^\ell a_i a_j \delta(A^n(p_i), p_j)
\end{equation}
where $\delta(\cdot, \cdot)$ is the Kronecker symbol. For instance, if $A^n(p_i)\neq p_j$ for all $i$, $j$, and $n\geq 1$, then 
$\deg(h^n)=\deg(g)^2$ for all $n\geq 1$. 

Now, let us organize the points $p_i$ with respect to the different orbits of $A$. This means that we index the points $p_i$ 
in such a way that there exist a sequence of integers $k_0=0< k_1 < \cdots < k_m=\ell$ and a sequence of points $q_1$, 
$q_2$, $\ldots$, $q_m\in \bbP^2(\bfk)$ satisfying 
\begin{itemize}
\item[(Orb. 1)] the points $p_{k_{j-1}+1}$, $\ldots$, $p_{k_j}$ lie on the orbit of $q_j$ under the action of $A$ ($\forall j\in \{1, \ldots, m\}$).

\item[(Orb. 2)] the points $q_j$ are in pairwise distinct orbits.
\end{itemize}
The subset $F_j:=\{p_{k_{j-1}+1}, \ldots, p_{k_j}  \}$ is therefore a finite part of the orbit $A^{\Z} (q_j)$; it coincides with the intersection of $A^\Z (q_j)$ with 
$\Ind(g)$. Now, assume that 
\begin{itemize}
\item[(Hyp. 2)] The multiplicity function $p_i\mapsto a_i$ is constant, and equal to some positive integer ${\overline{a_j}}$, on each $F_j$. 
\end{itemize}
Then, the formula for the degree of $h^n$ reads 
\begin{equation}
\deg(h^n)=\deg(g)^2-\sum_{j=1}^m {\overline{a_j}}^2 \vert A^n(F_j)\cap F_j\vert.
\end{equation}
To describe the sequence $n \mapsto \vert A^n(F_j)\cap F_j\vert $, assume that 
\begin{itemize}
\item[(Hyp. 3)] None of the $q_j$ is a periodic point of $A$.
\end{itemize}
Parameterizing each of those orbits by $\varphi_j\colon n\mapsto A^n(q_j)$, $F_j$ corresponds to a finite subset $G_j=\varphi_j^{-1}(F_j)$ of
$\Z$, and $A$ is conjugate to the 
translation $T\colon n\mapsto n+1$ along the orbit of $q_j$. Then,  
\begin{equation}
\vert A^n(F_j)\cap F_j\vert=\vert T^n(G_j)\cap G_j\vert =: ol_j(n)
\end{equation}
where $ol_j(n)$ is, by definition, the overlapping length of $G_j$ and $T^n(G_j)$. For instance, if $\vert G_j\vert =1$, then 
$ol_j(n)$ is equal to $1$ if $n=0$ and vanishes for $n\neq 0$. If $G_j=\{a,b\}$ and $b-a=s>0$, then $ol_j(n)$ is equal to 
$2$ when $n=0$, to $1$ when $n=s$, and vanishes on $\Z\setminus\{0, s\}$; in that case, the restriction of $ol_j$ to $\N^*$ is the Dirac
mass on $s$. Thus, we obtain 

\begin{lem}\label{lem:D-dirac}
Every function $D\colon \N^*\to \N$ with finite support $S$ is a sum of $\vert S\vert$ functions $ol_j(\cdot)$ for 
subsets $G_j\subset \N^*$ with exactly two elements each. 
\end{lem}

Now, adding the assumption 
\begin{itemize}
\item[(Hyp. 4)] If $\vert F_j\vert \geq 2$, then $\overline{a_j}=1$,
\end{itemize}
one gets 
\begin{equation}
\deg(h^n)=\deg(g)^2-\sum_{j=1}^m ol_j(n)
\end{equation}
and $ol_j(n)=0$ when $\overline{a_j}\neq 1$ (resp. when $\vert F_j\vert = 1$). 

To prove Theorem~\ref{thm:Oscillations}, we first apply Lemma~\ref{lem:D-dirac} to write $D=\sum_{j=1}^m ol_j(n)$ for some finite
subsets $G_j\subset \N^*$ with exactly two elements each. Then, we shall construct an element $g\in \Bir(\bbP^2_\bfk)$
with a unique indeterminacy point $q_0$ of multiplicity $>1$, and $m$ indeterminacy points $p_j$ of multiplicity $1$. 
All we have to do is to choose $g$ in such a way that the four hypotheses (Hyp. 1) to (Hyp. 4) are satisfied.

\subsection{Jonqui\`eres transformations}
Let $N$ be a positive integer.
Given a distinguished point $p_0\in \bbP^2(\bfk)$ (also denoted $q_0$ in what follows), and 
$2N$ points $p_1$, $\ldots$, $p_{2N}\in \bbP^2(\bfk)$, we consider the following incidence properties.
\begin{itemize}
\item[(Jon. $1$)] The finite set $\{p_0, p_1, \ldots, p_{2N}\}$ does not contain three colinear points. 
\item[(Jon. $k$)] A plane curve of degree $k$ which contains $q_0$ with multiplicity $k-1$ 
contains at most $2k$ of the remaining points $p_i$ ($1\leq i\leq 2N$).
\end{itemize}
Note that in the second property, we denote $p_0$ by $q_0$ to distinguish it from the other 
points. For example, condition (Jon. $3$) means that there is no cubic 
curve which passes through $q_0$ with multiplicity $2$ and contains $7$ of the remaining points
$p_1$, $\ldots$, $p_{2N}$. As we shall see, this is a generic property on the points $p_i$.
Condition  (Jon. $k$) is empty when $k>N$.

\begin{thm}[see \cite{Blanc-Cantat}]\label{thm:Jonq}
 If the conditions $($Jon. $1)$ and $($Jon. $k)$ are satisfied for 
all $2\leq k \leq N$, there exists a birational transformation $g\colon \bbP^2_\bfk\dasharrow \bbP^2_\bfk$
such that 
\begin{enumerate}
\item $g$ preserves the pencil of lines though $q_0$;
\item $g$ has degree $N+1$;
\item the base points of $g$ are $p_0=q_0$, with multiplicity $N$, and the points $p_1$, $\ldots$, $p_{2N}$, 
all with multiplicity $1$; 
\item $g^{-1}_\bullet\bfe_0=(N+1) \bfe_0 - N\bfe(q_0)-\sum_{j=1}^{2N} \bfe(p_j)$. 
\end{enumerate}
\end{thm}

The transformations provided by this theorem are ``Jonqui\`eres transformations'', because they preserve a pencil of lines. 

\begin{rem}\label{rem:dimensions} In what follows, we shall repeatedly use that the vector space of polynomial functions of degree $k$ in 
$3$ variables has dimension 
\begin{equation}
Dim(k)=\left(\begin{array}{c} k+2 \\ 2 \end{array} \right)= \frac{(k+2)(k+1)}{2}=Dim(k-1)+k+1,
\end{equation}
and that the polynomial equations corresponding to curves passing through a given point $q_0$ with 
multiplicity $(k-1)$ form a subspace of dimension 
\begin{equation}
dim(k)=2k+1.
\end{equation}
Also, note that two curves $C$ and $D$ in this system intersect in $q_0$ with multiplicity $(k-1)^2$, and in 
$2k-1$ extra points (counted with multiplicity). \end{rem}

\subsection{The construction}\label{par:construction-jonq}

To prove Theorem~\ref{thm:Oscillations} directly,  we adopt the strategy developped in Section~\ref{par:prelim-oscill}. 
First write $D=\sum_{j=1}^m ol_j$ for some subsets $G_j=\{a, a+s_j\}$ of $\N^*$ 
with $2$ elements each. Fix an algebraically closed field $\bfk$ and an infinite order element $A$ in $\PGL_3(\bfk)$. We want to construct
 $m+1$ points $q_0$, $q_1$, $\ldots$, $q_m$ such that
\begin{itemize}
\item[(Prop. 1)] the points $q_0$, $q_1$, $\ldots$, $q_m$ are on infinite, pairwise distinct orbits of $A$;
\item[(Prop. 2)] the $(2m+1)$ points $p_0=q_0$ and $p_{2j-j}= q_j$, $p_{2j}=A^{s_j}(q_j)$ for $1\leq j\leq m$ 
satisfy the conditions (Jon. $1$) and (Jon. $k$) for all $2\leq k \leq m$. 
\end{itemize}
Then, Theorem~\ref{thm:Jonq} provides a Jonqui\`eres transformation $g$ of degree $m+1$ with base points $p_0=q_0$ 
of multiplicity $m$ and $p_j$ of multiplicity $1$ for $1\leq j \leq 2m$. 
And the birational transformation $h=g\circ A \circ g^{-1}$ satisfies the conclusion of Theorem~\ref{thm:Oscillations} (with $d=(m+1)^2$).

In what follows, $A$ is fixed; the point $q_0$ is also fixed, and is chosen so that its orbit is infinite. 
We shall prove that a generic choice of the points $q_j$ satisfy (Prop. 1), and (Prop. 2). 

A word of warning: we use the word "generic" in an unconventional way. We say that the $q_j$ are {\bf{generic}} if $(q_j)_{j=0}^{m}$
is chosen in the complement of countably many proper, Zariski closed subsets  of $(\bbP^2_\bfk)^{m+1}$ 
that depend only on~$A$. A more appropriate terminology would be "very general", but we use 
"generic" and "generically" for short.

A generic choice of the $q_i$, $i\geq 1$, satisfies (Prop. 1), so we analyze only (Prop. 2). 
For that purpose, we prove recursively that a generic choice satisfies (Jon. $k$), adding progressively new
points $q_m$; the recursion is on $m$ (for each $k$). We start with condition (Jon. $1$). 

\subsubsection{Condition $($Jon. $1)$}\label{par:Cond-J1}

Here, we prove that {\sl{a generic choice of the $q_i$, $i\geq 1$ satisfies $($Jon. $1)$}}. 

First, choose $q_1$ in such a way 
that its orbit (under the action of $A$) is infinite and does not contain $q_0$; a generic choice of $q_1$ works. 
The extra condition imposed by $(Jon. $1$)$ is that $q_0$, $q_1$, and $A^{s_1}(q_1)$ are not colinear. But if
those three points were colinear for a generic choice of $q_1$, $A^{s_1}$ would preserve the pencil of lines
through $q_0$, $A^{s_1}(q_0)$ would be equal to $q_0$, and $q_0$ would have a finite orbit. 
Thus, (Jon. $1$) is satisfied for $p_0=q_0$, $p_1=q_1$, and $p_2=A^{s_1}(q_1)$ 
as soon as $q_1$ is chosen generically.

Choosing $q_2$ generically, its orbit is infinite, and disjoint from the orbits of $q_0$ and $q_1$.
We want to show that no three of the points $p_0=q_0$, $p_1=q_1$, $p_2=A^{s_1}(q_1)$, $p_3=q_2$, 
$p_4=A^{s_2}(q_2)$ are colinear. First, we choose $q_2$ in the complement of the lines connecting two of the   
points $p_0$, $p_1$, $p_2$ and their images by $A^{-s_2}$. Then, if three of those points are colinear, this means that there
is a point $p_j$ with $0\leq j \leq 2$ such that $p_j$, $q_2$, and $A^{s_2}(q_2)$ are colinear. Since $q_2$
is generic, this would imply that $A^{s_2}$ preserves the pencil of lines through  $p_j$. But then 
$p_j$ would have a finite orbit, in contradiction with the choice of $q_0$ and $q_1$.

The same argument works for the points $q_j$, $j\geq 3$: one just needs to choose $q_j$ generically in order to have an infinite orbit, 
to avoid the orbits of the previous points $p_i$ ($0\leq i\leq 2j-2$), and to avoid the lines connected two of the $p_i$ ($0\leq i\leq 2j-2$)
and their images by $A^{-s_j}$. This proves that (Jon. $1$) is satisfied for a generic choice. Thus, we may assume that it
is satisfied in what follows. 

\subsubsection{Condition $($Jon. $2)$} 

To warm up, we prove that a generic choice of the $q_i$ satisfies (Jon. $2$), i.e.  for conic curves.
We already know it satisfies (Jon. $1$). Thus, we want to prove that, for a generic choice of the 
$q_i$, a conic that contains $q_0$ with multiplicity $1$ does not contain more than $4$ of the $p_i$. 
First, we show that the following property is generic:
\begin{itemize}
\item[(Uniq. 2)] Five points $(q_0, p_{i_1}, \ldots, p_{i_4})$
with $i_j\in \{1, \ldots, 2m\}$ are always contained in a unique conic curve, and this curve is irreducible. 
\end{itemize}

\begin{proof}  Fix four indices $i_j$, and consider a conic curve $C$ containing  $q_0$, $p_{i_1}$, $\ldots$, and $p_{i_4}$. 
From Remark~\ref{rem:dimensions},
$Dim(2)=6$ and there is at least one such curve $C$. If there is a second one $D$, 
we get two distinct conics with at least $5$ common points, so that $C$ and $D$ share a common irreducible component. But a
reducible conic containing $(q_0, p_{i_1}, \ldots, p_{i_4})$ is made of two lines, and we obtain a  contradiction with (Jon. $1$).
\end{proof}

Now, let us prove that (Jon. $2$) is generically satisfied. As in Section~\ref{par:Cond-J1}, we do a recursion on the
number $m$ of points $q_i$. This condition is obviously satisfied for $m\leq  2$, since then at most five points are involved. 
Assume now that property (Jon. $2$) is satisfied for a generic choice of $(m-1)$ points. Let $q_m$ be an extra generic point. 
Consider the set of conics $C$ containing $q_0$ and four points among  the  $p_i$, for $i\leq 2(m-1)$, and their images by $A^{-s_m}(C)$.
By (Uniq. $2$) this set is finite; since $q_m$ is generic, we may assume that it is not on one of these conics. 
It remains to show that, for any choice of three indices $i_1$, $i_2$, $i_3\leq 2(m-1)$, the unique conic through 
$(q_0, p_{i_1}, p_{i_2}, p_{i_3}, q_m)$ does
not contain $A^{s_m}(q_m)$. Assume the contrary.  Since $q_m$ is chosen generically, then $A^{s_m}$ would 
fix the pencil of conics through the base points $(q_0, p_{i_1}, p_{i_2}, p_{i_3})$; it would permute these base-points, 
and the orbit of $q_0$ would not be infinite, a contradiction. 

To sum up, properties (Jon. $1$), (Jon. $2$) and (Uniq. $2$) are satisfied for a generic choice of the points $q_i$.

\subsubsection{Condition $($Jon. $k)$} 

Now, we prove by induction that (Jon. $k$) is gene\-rically satisfied for all $k\leq m$. As above, the points $p_i$ will be
$p_0=q_0$, and $p_{2i-1}=q_{i}$, $p_{2i}=A^{s_{i}}(q_{i+1})$ for $i\geq 1$.
Our recursion will simultaneously prove that (Jon. $k$) and the following condition (Uniq. $k$) are satisfied 
for a generic choice of the $q_i$.
\begin{itemize}
\item[(Uniq. $k$)] For every choice of $2k$ points $(p_{i_1}, \ldots, p_{i_{2k}})$
with $i_j\in \{1, \ldots, 2m\}$, there is a unique plane curve of degree $k$ containing $q_0$ with multiplicity $k-1$
and $(p_{i_1}, \ldots, p_{i_{2k}})$ with multiplicity $\geq 1$, and  this curve is irreducible. 
\end{itemize}

First, we show that {\sl{the properties $($Jon. $\ell)$ for $\ell\leq k$ imply $($Uniq. $k+1)$}}. 

For this, consider the 
system of curves of degree $k+1$ with 
multiplicity $k$ at $q_0$. By Remark~\ref{rem:dimensions}, their equations form a vector space of dimension $dim(k+1)=2k+3$. Now, 
fix a finite subset $P$ of the $p_i$ with $2(k+1)$ elements, and consider the subsystem of curves containing $P$. 
There is at least one such curve $C$ because $2k+3-2(k+1)=1$. If there were a second curve $C'$ of this type, 
then, counted with multiplicities, $C\cap C'$ would contain at least $k^2 + 2 (k+1)= (k+1)^2+1$ points. This is larger than
$\deg(C)\deg(C')$. So, either $C$ is unique, or $C$ is a reducible 
curve containing $P$ and passing through $q_0$ with multiplicity $k$. Then, one of the connected components
would contradict (Jon. $\ell$), for some $\ell \leq k$. Indeed, decompose $C$ into connected
components $C=C_1+\cdots + C_s$ of degree $d_i$, with $d_1+\cdots d_s=k+1$. If $C_i$ is not
a line, its multiplicity in $q_0$ is bounded by $mult(C_i;q_0)\leq d_i-1$. Thus, at most one of these components
has degree $\geq 2$. If $C$ is a union of $k+1$ lines, we obtain a contradiction with (Jon. $1$).
If $C=L_1+\cdots L_{s-1}+C_s$ with $\deg(C_s)\geq 2$, and none of the lines $L_i$ contains three points of $P$, 
then $C_s$ has degree $d_s=k+1-s$, contains $q_0$ with multiplicity $d_s-1=k-s$, and contains $2k+2-s$ points $p_i\in P$. 
From (Jon. $k+1-s$) we know that such a curve does not exist.

Now, we derive (Jon. $k+1$) from (Uniq. $k+1$). We do a recursion on the number $m$ 
of points. There is nothing to prove for $m\leq k$. Now, assume that (Jon. $k+1$) is generically satisfied for 
$m$ points;  the problem 
is to add one more point $q_{m+1}$ and verify that (Jon. $k+1$) is satisfied. Apply (Uniq. $k+1$): if $P$ is a finite
subset of the $p_i$ with $\vert P\vert =2(k+1)$, there is a unique curve $C_P$ of degree $k+1$ containing $P$ 
and passing through $q_0$ with multiplicity $k+1$. Since $q_{m+1}$ is chosen generically, we may assume that
it is not on one of the curves $C_P$ or $A^{-s_{m+1}}(C_P)$ for any $P\subset \{p_1, \ldots p_{2m}\}$ with $\vert P\vert=2(k+1)$. Now, 
(Jon. $k+1$) fails if and only if there is a finite subset $Q\subset \{p_0, \ldots p_{2m}\}$ with $2k+1$ elements
such that the linear system of curves of degree $k+1$ containing $q_0$ with multiplicity $k$ and $Q$ 
with multiplicity $1$ is $A^{s_{m+1}}$-invariant. By (Uniq. $k+1$), this linear system is a pencil with a
finite set of base points containing $Q$; then, the orbit of the points $p_j\in Q$ would be finite, contradicting (Jon. $1$). 

\subsection{Conclusion}\label{par:Jonq-conclusion} 
Let us come back to the proof of Theorem~\ref{thm:Oscillations}. What did we do in Section~\ref{par:construction-jonq} ? 
We fixed the function $D$. Then we proved that, if $\bfk$ is an algebraically closed field, and 
if $A$ is an element of $\PGL_3(\bfk)$ of infinite order, a very general choice of the points $q_0$, $q_1$, $\ldots$, $q_m$
satisfy (Prop.~$1$) and (Prop.~$2$). As a consequence, if $\bfk$ is uncountable, such a $(m+1)$-tuple does exist, and there is
a Jonqui\`eres transformation $g$ such that the birational transformation $h=g\circ A \circ g^{-1}$ satisfies the conclusion of Theorem~\ref{thm:Oscillations}.

\subsection{On the realization of bounded degree sequences}\label{par:realization-qbar}

In this section, our main goal is to prove that a bounded sequence of degrees $(\deg(f^n))_{n\geq 0}$ which is realizable 
by at least one birational map $f\in \Bir(\bbP^2_\C)$ is also realizable by some birational map $g\in \Bir(\bbP^2_{\overline{\Q}})$; we
also use this opportunity to show how to combine the adelic topology with $p$-adic arguments. 

\subsubsection{The adelic topology} \label{par:adelic_topology_summary}
Here we gather some of the properties satisfied by the {\bf{adelic topology}}, which was introduced in Section~3 of \cite{Xie2019}.
Let $\bfk$ be an algebraically closed extension of $\overline{\Q}$ of finite transcendence degree.
Let $X$ be a variety over $\bfk$. The adelic topology is a topology on $X(\bfk)$, which is defined by considering all 
possible embeddings of $\bfk$ in $\C$ and the $\C_p$, for all primes $p$. Assume, for instance, that $X$ is defined over
the prime field $\Q$. If $\iota\colon \bfk\to \C_p$ (resp. $\C$) is an 
embedding, it induces an embedding $\Phi_\iota\colon X(\bfk)\to X(\C_p)$ (resp. $X(\C)$); then,  given any open subset $U$ 
of $X(\C_p)$ for the $p$-adic (resp. euclidean) topology, the union $\cup_\iota \Phi_\iota^{-1}(U)$ for all embeddings $\iota\colon \bfk \to \C_p$ is, by definition, 
an open subset of $X(\bfk)$ in the adelic topology (see \cite{Xie2019}). 
The adelic topology has the following basic properties.
\begin{enumerate}
\item It is stronger than the Zariski topology. If $\dim(X)\geq 1$, there are non-empty, adelic, open subsets $\mathcal{U}$ and $\mathcal{U}'$ of $X(\bfk)$
such that $\mathcal{U}\setminus \mathcal{U}'$ is Zariski dense in $X$.
\item It is $T_1$, that is for every pair of distinct points $x, y\in X(\bfk)$ there are adelic open subsets $\mathcal{U}$, $\mathcal{V}$ of $X(\bfk)$ such that 
$x\in \mathcal{U}, y\not\in \mathcal{U}$ and $y\in \mathcal{V}, x\not\in \mathcal{V}$.
\item Morphisms between algebraic varieties over $\bfk$ are continuous in the adelic topology.
\item \'Etale morphisms are open with respect to the adelic topology.
\item The irreducible components of $X(\bfk)$ in the Zariski topology coincide with the irreducible components of $X(\bfk)$ in the  adelic topology.
\item Let $K$ be a subfield of $\bfk$ such that (a) its algebraic closure ${\overline{K}}$ is equal to~$\bfk$, (b) $K$  is finitely generated over $\Q$, and (c) $X$ is defined over $K$.
Endow the Galois group  $\Gal(\bfk/K)$ with its profinite topology. Then the action 
$\Gal(\bfk/K)\times X(\bfk)\to X(\bfk)$ is continuous for the adelic topology.
\end{enumerate}
From (5), when $X$ is irreducible, {\sl{the intersection of finitely many nonempty adelic open subsets of $X(\bfk)$ is nonempty}}.
We say that a property $P$ holds for an {\bf{adelic general point}} if there exists an adelic dense open subset $\mathcal{U}$ of $X(\bfk)$, 
such that $P$ holds for all points in $\mathcal{U}$.

\subsubsection{Sequences of incidence times} Let $\bfk$ be a field. Let $X$ be a variety, $f: X\to X$ be an automorphism, and $Z\subset X$ be a subvariety, all of them defined over $\bfk$. 
For $x\in X$, the set of times $n\in \Z$ for which the incidence $f^n(x)\in Z$ occurs is 
\begin{equation}
N(x,f,Z) = \{ n\in \Z\; ; \; f^n(x)\in Z\}.
\end{equation}

\begin{thm}\label{thmreali}
Let $\bfk$ be an algebraically closed extension of $\overline{\Q}$ of finite transcendence degree.
Let $X$ be a variety, $f: X\to X$ be an automorphism, and $Z\subset X$ be a subvariety, all three defined over $\bfk$. 
Let $V$ be any irreducible subvariety of $X$ and let $\eta$ be the generic point of $V$. 
Then $N(x,f,Z)=N(\eta,f,Z)$ holds for an adelic general point $x\in V(\bfk)$.
\end{thm}

In the following corollary, we keep the notations and hypotheses of this theorem. Recall that the {\bf{constructible topology}} on $X$
is a topology for which constructible sets form a basis of open sets; for instance, if $\eta$ is the generic point of an irreducible subvariety
$V$ of $X$, any open set containing $\eta$ contains a dense Zariski open subset of $V$; and any constructible set is both open and closed 
(see~\cite{EGA-IV-I}, Section~(1.9) and in particular (1.9.13)).

\begin{cor}\label{correali}
Let $\chi: X\to \Z$ be a function which is continuous with respect to  the constructible topology of $X$.
Then the property  ``$\chi(f^n(x))=\chi(f^n(\eta))$ for all $n\in \Z$'' holds
for an adelic general point $x$ of $V(\bfk)$. 
\end{cor}

\begin{proof}
Since $X$ is compact with respect to the constructible topology, $\chi(X)$ is compact and therefore finite.
For $r\in \chi(X)$, $\chi^{-1}(r)$ is both open and closed in the constructible topology; since it is compact,  it
is a finite union
\begin{equation}
\chi^{-1}(r)=\bigcup_{i=1}^{m_r}(Z^r_i\setminus S^r_i)
\end{equation}
where the $Z^r_i$ and $S^r_i$ are closed subvarieties of $X$.
Let $B$ be the set of subvarieties $\{Z^r_i, S^r_i \; ; \;  r\in \chi(X), i=1,\dots, m_r\}$.
For every $x\in X$, the sequence $(\chi(f^n(x)))_{n\in \Z}$ is determined by the sets $N(x,f,Z)$,  for $Z\in B$.
By Theorem \ref{thmreali}, for every $Z\in B$, there exists a dense
adelic open subset ${\mathcal{U}}_Z$ of $V(\bfk)$ such that for every $x\in {\mathcal{U}}_Z$, $N(x,f,Z)=N(\eta,f,Z)$. By Property~(5) of~\S\ref{par:adelic_topology_summary}, $\cap_{Z\in B}{\mathcal{U}}_Z$ is nonempty, and every $x$ in this intersection satisfies $N(x,f,Z)=N(\eta,f,Z)$ for each $Z\in B$. For such a point, $\chi(f^n(x))=\chi(f^n(\eta))$ for all $n\in \Z$. \end{proof}

\begin{cor}\label{correaliex}
Let $\bfk$ be an algebraically closed field of characteristic zero.
Let $X$ be a variety and $f: X\to X$ be an automorphism of $X$, both defined over $\bfk$. 
Let $\chi: X\to \Z$ be a function which is continuous with respect to the constructible topology of $X$.
Let $V$ be an irreducible subvariety of $X$ and $\eta$ be its generic point. Then there exists a point $x\in V(\bfk)$ such that $\chi(f^n(x))=\chi(f^n(\eta))$ for all $n\in \Z.$
\end{cor}
\begin{proof}
There exists a subfield $\bfk_0$ of $\bfk$ which is algebraically closed and of finite transcendence degree over $\overline{\Q}$ such that $X$, $f$, $\chi$ and $V$ are defined over $\bfk_0$. Replacing $\bfk$ by $\bfk_0$, the conclusion follows from Corollary~\ref{correali}.
\end{proof}

\begin{rem}\label{rem:2.8} The proofs of Theorem \ref{thmreali} and its corollaries hold when $f$ is a non-ramified endomorphism, if we consider only the forward orbit of $x$.
\end{rem}

\subsubsection{Proof of Theorem \ref{thmreali}}  After replacing $X$ by the Zariski closure of the set $\{f^n(\eta)\}_{n\in \Z}$, we may assume that $X$ is smooth at $\eta$. We leave the proof of the following lemma to the reader.
\begin{lem}\label{lembasop}Let $Z$ and $Z'$ be subvarieties of $X$ and let $\ell$ be a nonzero integer. Then for every $x\in X$,
\begin{enumerate} 
\item $N(x,f,Z\cap Z')=N(x,f,Z)\cap N(x,f,Z')$;
\item $N(x,f,Z\cup Z')=N(x,f,Z)\cup N(x,f,Z')$;
\item $N(x, f^\ell,Z)=-N(x,f^{-\ell},Z)$;
\item $N(x,f,Z)=\cup_{i=1}^{|\ell|-1}(i+|\ell|\times N(x,f^{|\ell|}, f^{-i}(Z)))$;
\item $N(f^\ell(x),f,Z)=N(x,f,Z)-\ell.$
\end{enumerate} 
\end{lem} 
By (4) of Lemma~\ref{lembasop}, we may assume $X$ to be irreducible, and 
by (1) and (2), we only need to treat the case where $Z$ is an irreducible hypersurface.
Since $X$ is smooth at $\eta$, we can pick a point $b\in V(\bfk)$ such that $X$ and $V$ are smooth at $b$.
Let $K$ be a finitely generated extension of $\Q$ such that  $\overline{K}=\bfk$ and $X$, $Z$, $V$, $f$, and $b$ are all defined over $K$.
There exists a projective 
variety $X_K\to \Spec(K)$ such that $X= X_K \times_{\Spec(K)} \Spec(\bfk)$ and  $f_K: X_K\to X_K$ such that 
$f=f_K\times_{\Spec (K)}\Id.$ 
There exists a subring $R$ of $K$ that is finitely generated over $\Z$
such  that $\Frac(R)=K$, and a model $\pi : X_R \to \Spec(R)$ which is projective over $\Spec(R)$ and whose generic fiber is $X_K$. 
Shrinking $\Spec R$, we may assume that
\begin{itemize}
\item[(i)] all fibers of $X_R$ are absolutely irreducible of dimension $d=\dim X$, 
\item[(ii)] $f$ extends to an automorphism $f_R: X_R\to X_R$ over $R$,
\end{itemize}
and, denoting by $Z_R$, $V_R$, $b_R$ the Zariski closures of $Z$, $V$, $b$ in $X_R$,
\begin{itemize}
\item[(iii)] all fibers of $Z_R$ and $V_R$ are absolutely irreducible,
\item[(iv)] $b_R$ is a section of $V_R\to \Spec R$,
\item[(v)] $\Spec R$ is smooth, and $Z_R$ and $X_R$ are smooth along $b_R$.
\end{itemize}

\begin{lem}[see \cite{Lech1953}, or Lemma 3.2 in \cite{Bell2006}]\label{lem:Lech-Bell}
Let $L$ be a finitely generated extension of $\Q$ and let $B$ be any finite subset of $L$. The set of primes
$p$ for which there is an embedding $L\to \Q_p$ that maps $B$ into $\Z_p$ has positive density
among the set of all primes. 
\end{lem}
By {\bf{positive density}}, we mean that the proportion of primes $p$  that satisfy the statement among the first $N$ primes
is bounded from below by a positive number if $N$ is large enough; we shall only use that this set is infinite. 

Denote by $\C_p^{\circ}$ the closed unit ball in $\C_p$, and by $\C_p^{\circ\circ}$ the open unit ball. 
Since $R$ is integral and finitely generated over $\Z$, Lemma~\ref{lem:Lech-Bell} provides
infinitely many primes $p \geq 3$ such that $R$ embeds into $\Z_p\subseteq \C_p^{\circ}$. This induces an embedding
$\Spec(\Z_p) \to \Spec(R)$. Set $X_{\C_p} := X_R \times_{\Spec(R)}\Spec(\C_p)$, and $f_{\C_p}:=f_R\times_{\Spec(R)} \Id.$
All fibers $X_y$, for $y\in \Spec(R)$, are absolutely irreducible and of dimension $d$;
hence, the special fiber $X_{\overline{\bfF}_p}$ of $X_{\C_p^{\circ}} \to  \Spec(\C_p^{\circ})$ is absolutely irreducible. Denote by $b_{\overline{\bfF}_p}\in X_{\overline{\bfF}_p}$ the specialization of $b_R$. It is a  smooth point of $X_{\C_p^{\circ}}$.

There exists $m\geq 1$ such that  $(f|_{X_{\overline{\bfF}_p}})^m(b_{\overline{\bfF}_p})=b_{\overline{\bfF}_p}$ and  $(df^m|_{X_{\overline{\bfF}_p}})_{b_{\overline{\bfF}_p}}=\Id$. 
(Here we just used that $f$ is not ramified, see Remark~\ref{rem:2.8}.)
Let $U$ be the open subset of $X(\C_p)$ consisting of points whose specialization is $b_{\overline{\bfF}_p}$; then $U\simeq (\C_p^{\circ\circ})^{d}$.
We have $f^m(U)=U$, and the orbit of every point of $U$ is well-defined.
The restriction of $f^m$ to $U$ is an analytic automorphism 
\begin{equation}
f^m|_U: (x_1,\dots,x_d)\mapsto (F_1,\dots, F_d)
\end{equation}
where the $F_n$ are analytic functions on $U$; more precisely, $F_n(x)=x_n+\sum_{I}a_I^nx^I$
with coefficients $a_I^n\in \C_p^{\circ\circ}\cap \Z_p$. 
We may assume that the coordinate of $b\in U$ is $(0,\dots,0)$.
Then, there exists $\ell\in \Q^{+}$, such that 
\begin{equation}
f^m|_U=\Id \mod p^{2/\ell},
\end{equation}
where by definition $f\vert_U=\Id  \mod p^\alpha$ if $\vert a_I^n \vert \leq \vert p\vert^{\alpha}$ for all multi-indices~$I$.
Set $W=\{(x_1,\dots,x_d)\in U\; ; \; |x_i|\leq \vert p\vert^{1/\ell}\}\simeq (\C^{\circ}_p)^d$; it is $f^m$-invariant and  
\begin{equation}
f^m|_{W}=\Id \mod p^{1/\ell}.
\end{equation}
After replacing $m$ by some of its multiples $km$, we get
\begin{equation}
f^m|_{W}=\Id \mod p.
\end{equation}
By Lemma \ref{lembasop}(4), we may replace $f$ by $f^m$ and suppose  $m=1$; doing so, $Z$ must be replaced by the union of the irreducible hypersurfaces $f^{-1}(Z), f^{-2}(Z), \dots,$ up to  $f^{-m+1}(Z)$. For simplicity, in the rest of the proof, we only consider $Z$, but the same argument applies to any irreducible hypersurface, hence to each $f^{-j}(Z)$.

By \cite[Theorem 1]{Poonen2014}, 
there exists an analytic action $\Phi: \C_p^{\circ}\times W\to W$ of
$(\C_p^{\circ},+)$ on $W$ such that for every $n\in \Z$,  $\Phi(n,\cdot)=f^{n}|_{W}(\cdot).$ 
Denote by $\pi_1: \C_p^{\circ}\times W\to \C_p^{\circ}$  and $\pi_2: \C_p^{\circ}\times W\to W$ the projections onto the first and second factors.
Observe that $Y:=\Phi^{-1}(Z)$ is an analytic hypersurface of $\C_p^{\circ}\times W$.

Replacing $f$ by a suitable iterate and $W$ by a smaller polydisc meeting~$V$, we may assume that every irreducible component $Y_i$, $i=1,\,\ldots,\,\ell$, of $Y$ is smooth and $\pi_2 |_{Y_i}: Y_i\to W$ is surjective; for otherwise
$Z$ is $f$-invariant, and then $N(\eta, f, Z)$ is non-empty if and only if $V\subset Z$, if and
only if every $x\in V$ satisfies $N(x,f,Z)=\Z$. Shrinking $W$ again, we may assume that all restrictions  $\pi_2|_{Y_i}: Y_i\to W$ are isomorphisms.
Define $\tau_i\colon W\to \C_p^{\circ}$ by $\tau_i:=\pi_1\circ\pi_2|_{Y_i}^{-1}.$
Permuting the indices, we suppose that $\tau_i\vert_V$ is non-constant for $1\leq i\leq s$ 
and $\tau_i\vert_V$ is constant for $s+1\leq i \leq \ell$.
If $x$ is a point of $W\cap V(\bfk)$, then 
\begin{equation}
N(x,f,Z)=\{ \tau_i(x)|\,\, i=1,\,\ldots,\,\ell\}\cap \Z\subseteq \{ \tau_i(x)|\,\, i=1,\,\ldots,\,\ell\}\cap \Z_p
\end{equation}
and for the generic point,
\begin{equation}
N(\eta,f,Z)=\{ \tau_i|\,\, i=s+1,\,\ldots,\,\ell\}\cap \Z,
\end{equation}
because $\tau_i$ is a constant (an element of $\C_p^\circ$) if and only if $i\geq s+1$.

Now, note the following: if $i\leq s$, $\tau_i$ is not constant and is therefore an open mapping;  
since $\Z_p$ is a nowhere dense and closed subset of $\C_p^\circ$, $\tau_i^{-1}(\C_p^\circ\setminus \Z_p)$ is open and dense in $W$.
Thus,  $U':=\cap_{i=1}^{s}\tau_i^{-1}(\C_p^\circ\setminus \Z_p)$ is open and dense in $W$, and for $x\in U'\cap V(\bfk)$ we get 
$N(x,f,Z)=N(\eta,f,Z)$. By  \cite[Remark 3.11]{Xie2019}, this equality holds for an adelic general point $x\in V(\bfk)$. This concludes the proof.

\subsubsection{Bounded degree sequences}\label{par:bds} Let $m\geq \ell\geq 0$ be  integers, and $\bfk$ be a field.
Denote by $R(\bfk,m)\subset (\N)^\Z$ (resp. $\Lambda_\ell(\bfk, m)\subset \R$) the set of sequences $(\deg (f^n))_{n\in \Z}$ (resp. dynamical degrees $\lambda_\ell(f)$) for $f\in \Bir(\bbP^m_\bfk)$.

\medskip

\noindent{\bf{Question.-}} Do we have $\Lambda_\ell(\bfk,m)=\Lambda_\ell(\overline{\Q},m)$, or even $R(\bfk,m)=R(\overline{\Q},m)$, for every algebraically closed field $\bfk$ of characteristic zero ? 

\medskip

Here, we answer positively this question for the subset  $R^b(\bfk,m)\subset R(\bfk,m)$ of sequences $(\deg (f^n))_{n\in \Z}$ which are {\sl{bounded}}.  Note that $(\deg (f^n))_{n\in \Z}$ is bounded if and only if  $(\deg (f^n))_{n\in \Z_{\geq 0}}$ is bounded. The last assertion of 
Theorem \ref{thm:Oscillations} follows from the following statement. 
 
\begin{thm}\label{thm:realbounded} 
If $\bfk$ is an algebraically closed field of characteristic zero, then $$R^b(\bfk,m)=R^b(\overline{\Q},m).$$
\end{thm}

\begin{proof}[Proof of Theorem \ref{thm:realbounded}]

Since $\bfk$ has characteristic zero and is algebraically clo\-sed, we may view it as an extension of $\overline{\Q}$.
In particular, $R^b(\overline{\Q},m)\subset R^b(\bfk,m)$.
Let $f\colon \bbP^m_{\bfk}\dashrightarrow \bbP^m_{\bfk}$ be a birational transformation for which the degree sequence
$ (\deg (f^n))_{n\in \Z}$ is bounded. We want to  show that this sequence is in $R^b(\overline{\Q},m)$.

There exists a $\overline{\Q}$-subalgebra $A$ of $\bfk$ such that $A$ is finitely generated over $\overline{\Q}$ and $f$ is defined over $A$.
Then we may assume $\bfk=\overline{\Frac A}$.
We view 
\begin{equation}
B:=\Spec A
\end{equation}
as a variety over $\overline{\Q}$, and we denote by $\eta$ its generic point.
There is a birational transformation $F: \bbP^m_B\dashrightarrow \bbP^m_B$ over $B$ whose restriction to the generic fiber is $f$.
For every point $b\in B$, denote by $F_b$ the restriction of $F$ to the special fiber at $b$. 
All we need to show is the existence of a point $b\in B(\overline{\Q})$ such that $\deg(F_b^n)=\deg(f^n)$ for all $n\in \Z.$

Since $(\deg f^n)$ is bounded,  we may apply Weil's regularization theorem (see~\cite{Weil:1955, Huckleberry-Zaitsev, Zaitsev:1995, Kraft:regularization} for proofs, and~\cite[\S~2]{Blanc:TrGroups} or~\cite[Theorem~2.5]{Cantat:Compositio} for further references).
\begin{equation}
H_{\mathcal{X}}:=\pi^{-1}\circ F\circ \pi\colon \mathcal{X}\to \mathcal{X}
\end{equation} 
is an automorphism on the generic fiber $X:=  \mathcal{X}_\eta$, where $\eta$ denotes the generic point of $B$. After shrinking $B$, we may assume that $B$ is smooth, $\pi$ is birational on all fibers and $H_{\mathcal{X}}$ is an automorphism over~$B$. 

Denote by $\Aut_B(\mathcal{X})$ the automorphism group scheme of $\mathcal{X}$ over $B$. Then $H_{\mathcal{X}}$ can be viewed as a section of $\Aut_B(\mathcal{X})(B).$ 
We view $\Aut(X)$ as the generic fiber of $\Aut_B(\mathcal{X})$ and denote by $H$ the element of $\Aut(X)$ induced by $H_{\mathcal{X}}$.

Since $(\deg (f^n))_{n\in\Z}$ is bounded, the closure $G$ of $\{H^n\; ; \; n\in \Z\}$ in $\Aut(X)$ is a commutative subgroup of $\Aut(X)$ of finite type. After shrinking $B$, we may assume that the closure $\mathcal{G}$ of $G$ in $\Aut_B(\mathcal{X})$ is a group scheme of finite type. We have  $H\in \mathcal{G}(B)$.

Denote by $\phi: \mathcal{G}\to B$ the structure morphism. 
Denote by $L_H: \mathcal{G}\to \mathcal{G}$ the automorphism $g\mapsto H_{\phi(g)}\cdot g$, where $\cdot$ is the product for the group structure.
View $\mathcal{G}$ as a variety over $\overline{\Q}$. The function $\delta\colon \mathcal{G}\to \N$ defined by 
$$
\delta : g\mapsto \deg(F_{\phi(g)}) = \deg((\pi\circ H_{\chi}\circ \pi^{-1})_{\phi(g)})
$$
is lower-semi continuous with respect to the Zariski topology; in particular, $\delta$ is continuous with respect to the constructible topology. 
View (the image of the section) $H$ as a subvariety of $\mathcal{G}$ over $B$ and denote by $\alpha$ its generic point.
Then, by construction, $\deg(f^n)=\delta(L_H^n(\alpha))$ for all $n\in \Z$.
All we need to show is the existence of a point $x\in F(\overline{\Q})$ such that $\delta(L_F^n(\alpha))=\delta(L_F^n(x))$ for all $n$: this follows from Corollary \ref{correaliex}.
\end{proof}

\begin{center}
{\bf{-- Part C. --}}
\end{center}

\subsection*{Introduction} Recall 
that there are three types (or classes) of birational transformations of $\bbP^2$: elliptic, parabolic, and loxodromic; and the parabolic class splits into the disjoint union of Jonqui\`eres twists and Halphen twists (see \S~\ref{par:Halphen_twists_basics}, or \cite{Cantat:Survey}). These
transformations can be characterized by their degree growth, $\deg(f^n)$ being bounded when $f$ is elliptic, or growing polynomially with $n$ 
when $f$ is parabolic (the growth is linear for Jonquières twists and quadratic for Halphen twists), or growing exponentially fast (like $\lambda_1(f)^n$) when $f$ is loxodromic. Also, loxodromic transformations do not preserve any pencil of curves, while parabolic ones preserve a unique pencil (of genus $0$ or $1$ for Jonquières and Halphen twists respectively), and elliptic ones preserve infinitely many pencils. 

Fix a degree $d\geq 1$, and consider the algebraic variety $\Bir_d(\bbP^2_\bfk)$ of birational transformations of degree $d$. 
Set 
\begin{align}
\notag Lox(d) & =  \{ g \in \Bir_d(\bbP^2_\bfk)\; ;\; g \;  {\text{ is loxodromic}} \}, 
\end{align}
and define similarly the subsets $\Hal(d)$, $\Jon(d)$,  and $\Ell(d)$ (see~Equations~\eqref{eq:hald} to~\eqref{eq:elld} below). 
Then, $\Bir_d(\bbP^2_\bfk)$ is the disjoint union 
$\Ell(d) \sqcup \Jon(d)\sqcup \Hal(d) \sqcup \Lox(d)$, $\Bir_1(\bbP^2_\bfk)$ coïncides with $\Ell(1)$, and by Theorem~\ref{thm:Xie}, $\Lox(d)$ is a dense open subset of $\Bir_d(\bbP^2_\bfk)$ for every $d\geq 2$.
The next sections study the structure of these subsets, the degrees of invariant pencils when such a pencil 
exists, {i.e.} for $f\notin \Lox(d)$, and the intersections of conjugacy classes of $\Bir(\bbP^2_\bfk)$ with $\Bir_d(\bbP^2_\bfk)$.  
Most of these questions where given a complete answer in \cite{Blanc-Cantat} 
for loxodromic transformations; so, here we focus on elliptic and parabolic ones.

The first result we shall prove says that {\sl{$\Hal(d)$ is a constructible subset of $\Bir_d(\bbP^2)$}} (see~Corollary~\ref{cor:HT-degree-bounds}), while in Examples \ref{eg:jonnotcons} and \ref{eg:jonnotconscarp}, we show that {\sl{$\Ell(d)$ and $\Jon(d)$ are not constructible 
in general for $d\geq 2$}}. Then, we obtain the following bound for degrees of invariant pencils (See Theorem~\ref{thm:bounds for halphen},~\ref{thm:bounds for jonquieres}, and~\ref{thm:bounds for elliptic}).

\medskip

\noindent {\bf{Theorem.}} 
{\emph{For $d\geq 1,$ there is a positive integer ${\mathrm{pen}}(d)$, such  every $f\in \Bir_d(\bbP^2_\bfk)$ which is not 
loxodromic preserves a pencil of degree $\leq  {\mathrm{pen}}(d)$. So, if $f$ is parabolic the degree of its unique invariant 
pencil is bounded by  ${\mathrm{pen}}(\deg(f))$, and if $f$ is elliptic  it preserves at least one pencil of 
degree $\leq {\mathrm{pen}}(\deg(f))$.}}

\medskip

This result may be considered as a positive answer to the Poincar\'e problem of bounding the degree of first integrals, 
but for birational transformations of the plane instead of algebraic foliations (see~\cite{Cerveau-Lins-Neto:1991, Lins-Neto:2002, Pereira:2002}). 

\medskip

\noindent {\bf{Example.}} For $a, b\in \bfk^*$, set $f: (x,y)\mapsto (ax,by)$. Then, for every pair of coprime integers $(c,d)$, the pencil $x^cy^d=t, t\in \bfk$, 
is preserved by $f$, so $f$ preserves pencils of arbitrary large degrees. If $f$ is an element of finite order, one can pull back pencils from
$\bbP^2_\bfk/\langle f\rangle$ to $\bbP^2_\bfk$ to construct invariant pencils of arbitrary high degree. 
Now, consider $f_a(ax,y)=(a x,x+y)$, for some $a\in \bfk^*$, then it preserves the pencil of affine lines through $(0,b)$ for any $b\in\bfk$. Over an algebraically closed field, any elliptic element of $\Bir(\bbP^2_\bfk)$ is conjugate to one of these three examples, so all of them preserve infinitely many pencils (see~\cite{Blanc-Deserti:2015}).

\medskip

Then, we study conjugacy classes
in $\Bir(\bbP^2_\C)$. What can be said of the elements $g$ which are in the closure of the conjugacy class of an element $f$ ?  How does it depend on the topology that one considers ? These questions are related to the following problem: {\sl{if $f$ and $g$ are conjugate, can we find a birational conjugacy $h\circ f\circ h^{-1}=g$ whose degree $\deg(h)$ is
controled by $\max(\deg(f),\deg(g))$?}} 
This last question has a positive answer when $f$ is loxodromic, this is 
one of the main results of \cite{Blanc-Cantat}; here, 
we answer this problem for the remaining classes.
 
 \medskip

\noindent {\bf{Theorem.}} 
{\emph{For every $d\geq 1$, there is an integer ${\mathrm{HalCo}}(d)$ such that, if
$f$ and $g$ are Halphen twists of degree $\leq d$, and if $f$ is conjugate to $g$ in $\Bir(\bbP^2_\bfk)$, 
then there is a birational map $h$ of degree $\leq {\mathrm{HalCo}}(d)$ such that 
$h\circ f \circ h^{-1}=g$. 
The conjugacy class of any Halphen twist is a constructible subset of $\Bir(\bbP^2_\bfk)$.}}

\medskip

As a consequence, for every Halphen twist $f$ and degree $d\geq 1$, the set $$\{g\in \Bir_d(\bbP^2)|\,\, g \text{ is conjugate to } f \text{ in } \Bir(\bbP^2) \}$$ is a constructible subset of $\Bir_d(\bbP^2).$
By the main results of \cite{Blanc-Cantat}, this property is also satisfied when $f$ is loxodromic, but as we shall 
see, it is not satisfied by Jonqui\`eres twists;  
we illustrate this in Theorem~\ref{thm:example-Jonq}. It is not satisfied by elliptic elements either (for instance, the conjugacy 
class of $(x,y)\mapsto (ax, by)$ contains $(x,y)\mapsto (ux, vy)$ for every point $(u,v)\in \bfk^\times \times \bfk^\times$ in the orbit of $(a,b)$ under the group of monomial transformations $\GL_2(\Z)$).

Section~\ref{par:limits} studies conjugacy classes but with respect to the euclidean topology, 
when the field is a local field. We provide an almost complete description of elements $g\in \Bir(\bbP^2_\bfk)$
of a given type (elliptic, parabolic, loxodromic) which are in the closure of the conjugacy class of an element of another type.
The only remaining case is the one of finite order elements of $\Bir(\bbP^2_\bfk)$ which are not conjugate to 
elements of $\Aut(\bbP^2_\bfk)$: for instance, can one write a Bertini involution as a limit of Jonqui\`eres or Halphen twists ?

All results mentionned so far will be proved in Section~\ref{par:pen_con_lim}, but to reach them, we need a detailed understanding of the 
automorphism groups $\Aut(X)$ of Halphen surfaces, and in particular some quantitative estimate regarding their conjugacy classes. 
This is done in Section~\ref{par:Halphen-Surfaces}, the goal of which is the following:

  \medskip

\noindent {\bf{Theorem.}} (See Theorem~\ref{thm:Conjugacy-in-Aut-Halphen})
{\emph{Let $X$ be a Halphen surface of index $m$.
Let $f$ and $g$ be two automorphisms of $X$.
If $f$ and $g$ are in the same conjugacy class of $\Aut(X)$, there exists an element 
$h$ of $\Aut(X)$ such that $h\circ f \circ h^{-1}=g$ and $
\deg(h):=(\bfe_0\cdot h^*\bfe_0)\leq A(m)(\deg(f)+\deg(g)),$ for some constant $A(m)$ that depends only
on $m$.}}

%
%
\section{Automorphisms of Halphen surfaces}\label{par:Halphen-Surfaces}
%
%

Sections~\ref{par:Halphen_surfaces} and~\ref{par:Halphen_NS} summarize and rephrase classical results concerning 
Hal\-phen surfaces, their N\'eron-Severi lattice, and their group of automorphisms: we refer to \cite{Dolgachev, Cantat-Dolgachev, Gizatullin:1980, Grivaux, Iskovskikh-Shafarevich} for proofs of these results. This is applied in Sections~\ref{par:Autt} and~\ref{par:Thm_Halphen} to control conjugacy classes and prove Theorem~\ref{thm:Conjugacy-in-Aut-Halphen}.

\subsection{Halphen pencils and  Halphen surfaces} \label{par:Halphen_surfaces}

\subsubsection{Halphen pencils $($see \cite{Dolgachev, Cantat-Dolgachev}$)$}\label{par:quasi-elliptic}
A Halphen pencil $\HP$ of index $m$ is a pencil of plane curves of degree $3m$, with $9$ base points of multiplicity $\geq m$ and no fixed component. 
We shall denote by $p_1$, $p_2$, $\ldots$, $p_9$
the nine base points (some of the $p_i$ may be infinitely near).  
Blowing up them, we get a rational surface  $X$ with Picard number $10$, together with a birational morphism $\epsilon\colon X\to \bbP^2_\bfk$; 
the pencil $\HP$ determines a 
fibration $\pi\colon X\to \bbP^1_\bfk$, with  fibers of arithmetic genus~$1$. Going back to $\bbP^2_\bfk$, $\pi$ corresponds to 
a rational function ${\overline{\pi}}=\pi\circ\epsilon^{-1}\colon \bbP^2_\bfk\dasharrow \bbP^1_\bfk$ whose fibers are the members of the pencil $\HP$. 
When $m>1$, $\pi$ has
a unique multiple fiber, of multiplicity $m$; it corresponds to a cubic curve with multiplicity $m$ in the pencil $\HP$ (see~\cite{Gizatullin:1980}). 
 In what follows, $\HP$ will always denote a Halphen pencil, and $X$ the associated Halphen surface. 
 
 \smallskip
 
{\bf{Warning.}} In characteristic $2$ or $3$, the fibration $\pi$ may be quasi-elliptic, all fibers being singular curves with a node or a cusp. 
{\sl{By convention}}, we shall always assume that the general fibers are smooth (so $\pi$ is a genus $1$ fibration whose 
generic fiber is smooth; we exclude quasi-elliptic fibrations: see Section~\ref{par:Halphen_twists_basics} which explains why we do so). 

\subsubsection{Automorphisms of the pencil $($see~\cite{Gizatullin:1980, Grivaux, Iskovskikh-Shafarevich}$)$}\label{par:rel-mini}
Let $\Bir(\bbP^2_\bfk; \HP)$ be the group of birational transformations of the plane preserving the pencil $\HP$. 
After conjugacy by the blow-up $\epsilon\colon X\to \bbP^2_\bfk$,  $\Bir(\bbP^2_\bfk; \HP)$ becomes a group of birational transformations of $X$ that
permutes the fibers of $\pi$. The fibration $\pi$ is relatively minimal, which amounts to say that there
is no exceptional divisor of the first kind in the fibers of $\pi$; this implies that $\Bir(\bbP^2_\bfk; \HP)$ is contained in $\Aut(X)$.
Moreover, the members of the linear system $\vert -mK_X\vert$ are given by the fibers of $\pi$; hence, $\Aut(X)$ permutes the
fibers of $\pi$, and we conclude that 
\begin{equation}
\epsilon^{-1}\circ \Bir(\bbP^2_\bfk; \HP)\circ \epsilon=\Aut(X)
\end{equation}
and there is a homomorphism $\tau\colon \Aut(X)\to \Aut(\bbP^1_\bfk)$ such that 
\begin{equation}
\pi \circ f = \tau(f)\circ \pi \quad \;  (\forall f \in \Aut(X)).
\end{equation}
Equivalently, there is a homomorphism $\tau\colon \Bir(\bbP^2; \HP)\to \Aut(\bbP^1)$
such that ${\overline{\pi}}\circ f=\tau(f)\circ {\overline{\pi}}$ for every $f\in \Bir(\bbP^2; \HP)$. 

\subsection{The N\'eron-Severi group and its geometry} \label{par:Halphen_NS}
 
\subsubsection{Picard group $($see~\cite{Cantat:Survey, Cantat-Dolgachev, Grivaux}$)$}\label{par:basis_for_NS}
The Picard group of $X$ is discrete, and coincides with the N\'eron-Severi group $\NS(X)$. 
A basis is given by the class $\bfe_0$ of a line (pulled back to $X$ by $\epsilon$)
and the nine classes $\bfe_i$ corresponding to the (total transform of the) exceptional 
divisors of the blow-ups of the $p_i$. The intersection form satisfies $\bfe_0^2=1$, $\bfe_i\cdot \bfe_j=0$ if $i\neq j$, and 
$\bfe_i^2=-1$ if $i\geq 1$. So, with this basis, the lattice $\NS(X)$ is isometric to $\Z^{1,9}$. We set $\NS(X;\R):=\NS(X)\otimes_\Z \R$
and view $\NS(X)$ as a lattice of $\NS(X;\R)$.
The canonical class of $X$ is the vector $k_X=-3\bfe_0+\bfe_1+\cdots +\bfe_9$. The anticanonical class 
\begin{equation}
\xi:=-k_X=3\bfe_0-\bfe_1-\cdots -\bfe_9
\end{equation}
will play an important role; the class of a general fiber of $\pi$ is equal to $m\xi$. 

\subsubsection{The spaces $\xi^\perp$, $P$ and $Q$.} The orthogonal complement $\xi^\perp$ of $\xi$ for the intersection form is a subgroup  
of $\NS(X)$ of rank $9$ that
contains $\xi$; we denote by $\xi^\perp_\R$ the real vector subspace of dimension $9$ in $\NS(X;\R)$ spanned by $\xi^\perp$. Denote by $P$ the quotient space $\xi^\perp_\R/\langle \xi\rangle_{\R}$, where  $\langle \xi\rangle_{\R}=\R\cdot \xi$ is the subspace of $\xi^{\perp}_\R$ spanned by $\xi$, and by $q_P\colon \xi^\perp_\R\to P$ the projection.
The intersection form induces a negative definite quadratic form on $P$; 
since $\xi$ is primitive,  $\xi^\perp/\xi$ is isomorphic to a lattice $P(\Z)\subset P$.
We denote by $\parallel \cdot \parallel_P$  the euclidean norm given by 
this quadratic form:
\begin{equation}
\parallel q_P(w) \parallel_P=(-w\cdot w)^{1/2}, \quad \forall w \in \xi^\perp_\R.
\end{equation}

Let $Q_0$ be the affine space 
\begin{equation}
Q_0=\bfe_0+\xi^\perp_{\R}\subset \NS(X;\R);
\end{equation}
then $Q_0(\Z)=\bfe_0+\xi^\perp=Q_0\cap \NS(X)$ is a lattice in $Q_0$. 
Denote by $Q$ the quotient of $Q_0$ by the $(\R,+)$-action $x\mapsto x+t\xi$, $t\in \R$, and
by $q_Q:Q_0\to Q$ the quotient morphism; set 
\begin{equation}
\bfe=q_Q(\bfe_0)  \quad {\text{ and }} \quad Q(\Z)=q_Q(Q_0(\Z)).
\end{equation} 
 Then $Q(\Z)=\bfe+\xi^\perp/\xi$ is a lattice of $Q$.
Observe that the composition of the translation $w\in \xi^\perp_\R\mapsto w+\bfe_0$ and the projection $q_Q$ induces a natural isomorphism $w\in P\mapsto w+\bfe\in Q$. Using this isomorphism to transport the metric $\parallel \cdot \parallel_P$ on $Q$, we get a natural metric 
$\parallel \cdot \parallel_Q$: for  $v_1,v_2\in Q$ and $w_i\in q_Q^{-1}(v_i), i=1,2$, we have 
\begin{equation}\label{equsqd}
\parallel v_1-v_2\parallel_Q^2=-(w_1-w_2)^2
\end{equation}
where the square, on the right hand side, denotes self-intersection. 

\subsubsection{Horospheres $($see~\cite{Cantat:Survey, Cantat-Guirardel-Lonjou}$)$} Let $\Hyp$ denote the half-hyperboloid 
\begin{equation}
\Hyp=\{u\in \NS(X;\R)\; ; \; u^2=1 {\text{ and }} u\cdot \bfe_0 > 0\}.
\end{equation}
With the riemannian metric induced by (the opposite of) the intersection form, $\Hyp$ is a model of 
the hyperbolic space of dimension $9$. The isotropic vector $\xi$ determines a point of the boundary $\partial\Hyp$.
By definition, the set 
\begin{equation}
\Hor=\{u\in \Hyp\; ; \; u\cdot \xi = 3\}=\{ u\in Q_0\; ; \; u^2=1\}
\end{equation}
is the {\bf{horosphere}} 
centered at $\xi $ that contains $\bfe_0$. We have
\begin{equation}
\Hor=\{ u = \bfe_0+ v\in \NS(X;\R)\; ; \; 2\bfe_0\cdot v+v^2=0,  {\text{ and }} v\cdot \xi=0\}.
\end{equation}
Observe that the restriction $q_Q|_{\Hor}: \Hor\to Q$ is a homeomorphism. In what follows, we denote its inverse by 
\begin{equation} 
\ph:=(q_Q|_{\Hor})^{-1}.
\end{equation}

\begin{rem}
The projection $q_P$ induces a linear isometry $\xi^\perp\cap \bfe_0^\perp\to P$.
If we compose it with the isomorphism $w\mapsto \bfe+w$ from $P$ to $Q$, and then apply $\ph$, 
we get the following parametrization   of the horosphere $\Hor$:
\begin{equation}\label{eq:para-hor}
w\in \xi^\perp\cap \bfe_0^\perp\mapsto \ph(\bfe+ q_P(w))=\bfe_0+w-\frac{w^2}{6}\xi .
\end{equation}
\end{rem}

\subsubsection{Isometries fixing $\xi$} 

If $\varphi$ is an isometry of $\NS(X;\R)$ we set 
\begin{equation}
\deg(\varphi)=\bfe_0\cdot \varphi(\bfe_0)
\end{equation}
and call it the {\bf{degree}} of $\varphi$.
When $\varphi=f_*$ for some  $f\in \Aut(X)$, then $\deg(\varphi)$ coincides
with the degree of the birational transformation $\epsilon \circ f \circ \epsilon^{-1}$; we denote
this number by $\deg(f)$ or by $\deg(f_*)$  and call it the degree of $f$.

If $\varphi$ fixes $\xi$, then $\varphi$ preserves $\Hyp$, the horosphere $\Hor$, and the affine subspace $Q_0$;
it induces an affine isometry $\varphi_Q$ of $Q$ such that 
\begin{equation}
\ph^{-1}\circ \varphi|_{\Hor}\circ \ph=\varphi_Q.
\end{equation}
Since $\Hor$ generates $\NS(X;\R)$, the isometry $\varphi_Q$ determines $\varphi$ uniquely:

\begin{lem}
The homomorphism $\varphi\mapsto \varphi_Q$, from the stabilizer of $\xi$ in the group $\Isom(\NS(X;\R))$ to the group 
of affine isometries of $Q$, is faithful. 
\end{lem}

The group  $\Aut(X)$ preserves the anti-canonical class $\xi$, so every element $f\in \Aut(X)$ 
determines such an affine isometry $f_Q$ of $Q$ as well as a linear isometry $f_P$ of $P$.
The lattices $Q(\Z)$ and $P(\Z)$ are invariant by these isometries. Hence,
$\Aut(X)$ is an extension of a discrete group
of affine isometries of $Q$ by the kernel of the representation 
$f\mapsto f_*$ of $\Aut(X)$ on $\NS(X)$. 

For every $v\in Q$, set $d(v)=\bfe_0\cdot \ph(v).$ 
Equations~\eqref{equsqd} and~\eqref{eq:para-hor} imply
\begin{equation}\label{eqdv}
d(v)=\bfe_0\cdot \ph(v)=1+\frac{1}{2}\parallel v-\bfe\parallel_Q^2.
\end{equation}
Then for every $f\in \Aut(X)$, we get $\deg(f)=d(f_Q(\bfe))$.

\subsubsection{Reducible fibers, and the subspace $R(X)$ $($see~\cite{Gizatullin:1980, Grivaux}$)$} \label{par:description-Irr}
Consider the finite set $\Irr(X)\subset \NS(X)$ of classes $[C_i]$ of irreducible components of fibers of $\pi$; for instance, $m\xi\in \Irr(X)$ (take a regular fiber);  if $m>1$ and $F=mF_0$ is the unique multiple 
fiber, we get $[F_0]=\xi\in \Irr(X)$, so $\xi\in \Irr(X)$ whatever the value of $m$. Taking two disjoint fibers, we obtain $[C_i]\cdot \xi=0$ for every $[C_i]\in \Irr(X)$. 

Let $R(X)\subset \NS(X)$ (resp. $R(X)_{\R}\subset \NS(X;\R)$) be the linear span of $\Irr(X)$.

If $F$ is a fiber, and $F=\sum_{i=1}^{\mu(F)} a_i C_i$ is its decomposition into irreducible components -- with 
$\mu(F)$ the number of such components --,
then there is a linear relation with positive coefficients
\begin{equation}\label{eq:linear-relation-F}
[F]=\sum_{i=1}^{\mu(F)} a_i [C_i]=m\xi
\end{equation}
between the elements $\xi$ and $[C_i]$ of $\Irr(X)$. It turns out that
all linear relations between the elements of $\Irr(X)$ 
are linear combinations of the relations given in Equation~\eqref{eq:linear-relation-F}, as $F$ describes the set 
of reducible fibers. 
So,
\begin{itemize}
\item[(1)] $R(X)$ is a free abelian group; its rank is equal to $1+\sum_F (\mu(F)-1)$, where the sum is over all reducible fibers;
\item[(2)] $\xi\in R(X)\subseteq \xi^{\perp}$.
\end{itemize}

The genus formula and the equality $\xi=-k_X$ show that there are only three possibilities for 
the elements $c$ of $\Irr(X)$: either $c=\xi$, or $c=m\xi$ (in these cases  $c^2=0$), or $c$ is the class of
a smooth rational curve of self-intersection $c^2=-2$. In fact, by the Hodge index theorem, an element $v$ of $R(X)$ satisfies $v^2=0$ 
if and only if it is contained in $\Z\xi$. This gives
\begin{itemize}
\item[(3)] every class $c\in \Irr(X)$ satisfies $c^2\in \{0, -2\}$, with $c^2=0$ if and only if $c=\xi$ or $m\xi$.
\end{itemize}
Then, using an Euler characteristic argument and Kodaira's classification of singular fibers, one can show
that there are at most $12$ singular fibers, each of them containing at most $10$ irreducible components (see~\cite{Cossec-Dolgachev:book}). 
If $C_i$ is such an irreducible component, either $C_i$ comes from a component of the pencil
$\HP$, and $1\leq \deg([C_i]):=[C_i]\cdot \bfe_0\leq 3m$, or $C_i$ comes from the blow-up of an infinitely near point, 
and $\deg([C_i]):=[C_i]\cdot \bfe_0=0$. So, we obtain
the existence of an integer $B_1(m)$, that depends only on the index $m$ of the Halphen pencil, and an integer $B_0$ 
(that does not depend on $X$ or $m$) such that 
\begin{itemize}
\item[(4)] the set $\Irr(X)$ has at most $B_0$ elements;
\item[(5)] every  class $c\in \Irr(X)$ has degree at most $B_1(m)$, i.e. $0\leq c\cdot \bfe_0\leq B_1(m)$.
\end{itemize} 

\begin{lem}\label{lemrfi}
Let $m$ be a positive integer. Consider the set of all Halphen surfaces $X$ of index $m$:
there are only finitely many possibilities for the subset $\Irr(X)\subset \Z^{1,9}=\NS(X)$
and for its linear spans $R(X)$ and $R(X)_{\R}$.
\end{lem}

Here, $\NS(X)$ is identified with $\Z^{1,9}$ by our choice of basis $\bfe_0$, $\ldots$, $\bfe_9$, as in Section~\ref{par:basis_for_NS}.

\begin{proof}
  According to Properties (2) and (4), 
\begin{equation}
\Irr(X)\subset \{\xi, m\xi\}\cup \{v\in \NS(X)\; ; \; v^2=-2 \; {\text{and }} \, v\cdot \bfe_0\leq B_1(m)\}.
\end{equation}
This set is finite, bounded by a uniform constant that depends on $m$ but not on $X$ (or $\HP$). 
The lemma follows. 
\end{proof}

\subsubsection{The cylinder $Q^+$} \label{par:cylinderQ+}
Now, let $c$ be an element of $\Irr(X)\subseteq \xi^{\perp}$. Since $c\cdot \xi=0$, the linear map 
$w\in \NS(X;\R)\mapsto w \cdot c\in \R$ induces an affine function on $Q$:  if $v=q_Q(w)$ for some $w$ in 
$Q_0$, then we set $v\cdot c=w\cdot c$ and this does not depend on the choice of $w$.
Set 
\begin{equation}
Q(c^\perp)=\{ u\in Q\; ; \; u\cdot c=0\} \; {\text{ and }} \;  Q(c^+)=\{ u\in Q\; ; \; u\cdot c\geq 0\}.
\end{equation}
If $c\in \Z_{>0}\xi$ then $Q(c^\perp)=\emptyset$   and $Q(c^+)=Q$.
If $c\in \Irr(X)$ satisfies $c^2=-2$, then $Q(c^\perp)$ is a hyperplane of $Q$ and $Q(c^+)$ is one of the two closed half-spaces
bounded by $Q(c^\perp)$.

We want to describe 
\begin{equation}
Q^+=\bigcap_{c\in \Irr(X)} Q(c^+).
\end{equation}
Since $c$ is represented by an effective curve, the point $\bfe=q_Q(\bfe_0)$ is in $Q^+$. 

Consider the spaces $A_0\subset \xi^\perp_\R$ and $A\subset P$ defined by 
\begin{equation}
A_0=R(X)_\R^\perp \quad {\text{and}} \quad A=A_0/\langle \xi\rangle_\R.
\end{equation}
They act respectively by translations on $Q_0$ and $Q$, and we denote by $q_{Q/A}\colon Q\to Q/A$ the affine quotient map associated to this action. Note that every fiber of $\pi$ is irreducible if and only if $R(X)_\R=\langle \xi\rangle_\R$, if and only if
$A=\xi^\perp_\R/\langle \xi\rangle_\R$, if and only if $A$ acts transitively on $Q$, if and only if $Q/A=0$.

\begin{lem}
The set $Q^+$ contains $\bfe$ and is invariant under the action of $A$ by translations. 
Its projection $q_{Q/A}(Q^+)$ is a compact polytope $D(X)$ in $Q/A$; the number of faces
of $D(X)$ is uniformly bounded $($by a constant $B_0$ that does not depend on $X$ and $m)$.
\end{lem}

\begin{proof}
The first assertion follows from the inclusion $A_0\subset R(X)_\R^\perp$. 

For the second, we shall use the following fact. Let $E$ be a finite dimensional vector space, and let $(L_i)_{1\leq i\leq s}$ be a
finite generating subset of the dual space $E^\vee$. Let $A_i$, $1\leq i \leq s$, be real numbers, and 
let ${\mathcal{C}}\subset E$ be the convex set determined by the 
affine inequalities 
\begin{equation}
\ell_i(x):=L_i(x)+A_i\geq 0.
\end{equation}
Assume that $\sum_i a_i\ell_i=0$ for some positive real numbers $a_i$. Then ${\mathcal{C}}$ is compact or empty.
In our case, the $\ell_i$ are given by the inequalities $[C_i]\cdot x \geq 0$ in $Q/A$. Consider a reducible fiber $F$ of $\pi$, and decompose it into irreducible components 
$F=\sum_{i=1}^{\mu(F)}a_i C_i$, with $[C_i]\in R(X)$. All $a_i$ are $>0$, and  
\begin{equation}
\sum_{i=1}^{\mu(F)}a_i([C_i]\cdot v)=m\xi \cdot v = 0
\end{equation} 
for every $v$ in $Q$. Since $D(X)$ contains $\bfe$ it is not empty, and we conclude that  $D(X)$ is compact. 

Since $Q^+$ is the intersection of the half spaces $Q(c^+)$ and $\sharp \Irr(X)\leq B_0$ is uniformly bounded (see \S~\ref{par:description-Irr}, Property~(4)), $D(X)$ has at most $B_0$ faces.
\end{proof}

Denote by ${\mathrm{Cent}}(D(X))$ the center of mass of $D(X)$ (for the Lebesgue measure on $Q/A$): 
${\mathrm{Cent}}(D(X))$ depends only on the affine structure of $Q/A$.
Set 
\begin{equation}
\Axis=q_{Q/A}^{-1}({\mathrm{Cent}}(D(X)))\subset Q;
\end{equation}
it is an affine subspace of $Q$ of dimension $\dim \Axis=\dim A=9-\rk(R(X)).$
\begin{pro}\label{pro:actpre}
The convex set $Q^+$ is a cylinder with base  $D(X)$ and axis $\Axis$. 
The action of $\Aut(X)$ on $Q$ preserves $Q^+$ and  $\Axis$.
\end{pro}

\begin{proof}
The group $\Aut(X)$ permutes the elements of $\Irr(X)$, so it 
preserves $Q^+$, and on $Q/A$ it preserves the polytope $D(X)$;
since it acts by affine automorphisms, it preserves its center of mass. 
\end{proof}

For every $f\in \Aut(X)$, denote by $f_{Q/A}$ its action on $Q/A$ and by $f_\Axis$ its action on $\Axis$. 
Denote by $p_\Axis: Q\to \Axis$ the orthogonal projection. 
Since $f_Q$ is an isometry, we have $p_\Axis\circ f_A=f_\Axis\circ p_\Axis$. 

Since $D(X)$ is compact, there exists $r\geq 0,$ such that $\parallel{p_\Axis(x)-x }\parallel_Q\leq r$ for every $x\in Q^+$.
This $r$ only depends on $\Irr(X).$
Then for every $x,y\in Q^+$, we have 
\begin{equation}\label{ieqpd}
\parallel x-y\parallel_Q^2\leq \parallel p_\Axis(x)-p_\Axis(y)\parallel_Q^2+4r^2.
\end{equation}

\subsection{The subgroup $\Aut^t(X)$}\label{par:Autt}

\subsubsection{Vertical translations $($see~\cite{Cantat-Dolgachev, Grivaux}$)$}\label{par:vertical-translations}

The generic fiber $\mathfrak{X}$ of $\pi$ is a genus $1$ curve over the field 
$\bfk(\bbP^1_\bfk)$. 
If we take an element $u$ in $\Pic(X)=\NS(X)$ and restrict it to the generic fiber, we get an element of $\Pic(\mathfrak{X})$; 
this restriction has degree $0$ if and only if $u$ is in $\xi^\perp$. This defines a homomorphism 
$\Tr\colon \xi^\perp\to \Pic^0(\mathfrak{X})(\bfk(\bbP^1_\bfk)).$ It is surjective, and its kernel coincides with the subgroup $R(X)$
of $\NS(X)$.
So, we get an isomorphism of abelian groups 
\begin{equation}
\tr\colon \xi^\perp/R(X) \to \Pic^0(\mathfrak{X})(\bfk(\bbP^1_\bfk))
\end{equation}
(we use the notation $\tr$ after taking the quotient by $R(X)$).  
Now, start with an element $q$ of   $\Pic^0(\mathfrak{X})(\bfk(\bbP^1_\bfk))$. On a general fiber $X_t$ of $X$, with $t\in \bbP^1_\bfk$, we get
a point $q(t)\in \Pic^0(X_t)$, and then  a translation $z\mapsto z+q(t)$ of the elliptic curve $X_t$. This gives a birational map $T(q)$ of
$X$ acting by translations along the general fibers; since the fibration $\pi\colon X\to \bbP^1_\bfk$ is relatively minimal,
this birational map is in fact an automorphism (see~\S~\ref{par:rel-mini}). So, we obtain a homomorphism 
\begin{equation}
T\colon \Pic^0(\mathfrak{X})(\bfk(\bbP^1_\bfk))\to \Aut(X).
\end{equation}
In what follows, $\iota: \xi^\perp/R(X)\to \Aut^t(X)$ will denote the composition $T\circ \tr$, 
as well as $T\circ \Tr$, with the same notation since this should not induce any confusion. 
Also, recall that $\tau\colon \Aut(X)\to \Aut(\bbP^1_\bfk)$ is defined in \S~\ref{par:rel-mini}.

\begin{pro}\label{pro:T} The homomorphism $T$ satisfies the following properties.
\begin{enumerate} 
\item $T$ is injective.
\item Its image is the normal subgroup $\Aut^t(X)\subset \Aut(X)$ of automorphisms $f$ such that
\begin{itemize}
\item[(i)] $f$ preserves each fiber of $\pi$, i.e. $\tau(f)=\Id_{\bbP^1_\bfk}$;
\item[(ii)] if $X_t$ is a smooth fiber, then $f$ acts as a translation on it.
\end{itemize}
\item If  $\alpha \in \NS(X)$ and $\gamma\in \xi^\perp$, then $\iota(\gamma)_*\alpha=\alpha+(\alpha\cdot \xi)\gamma \mod\,\, R(X)$ $($i.e.
$(T(\Tr(\gamma)))_*\alpha=\alpha+(\alpha\cdot \xi)\gamma \mod\,\, R(X))$.\end{enumerate}
\end{pro}

\begin{proof}
The homomorphism $T$ is injective because $q(t)\neq 0$ for at least one $t$ if $q\in \Pic^0(\mathfrak{X})(\bfk(\bbP^1_\bfk))\setminus \{0\}$.
Its image is contained in the normal subgroup $\Aut^t(X)$; it coincides with it, because if $f$ is an element of $\Aut^t(X)$, and $X_t$ is any 
smooth fiber, one can set $q(t)=f(x)-x$ for any $x\in X_t$, and this defines an element of $\Pic^0(\mathfrak{X})(\bfk(\bbP^1_\bfk))$ for which $T(q)=f$. 

So, we only need to prove (3). By  linearity, we may assume that $\alpha$ is represented by an irreducible curve $V\subset X$.
Denote by $V_{\eta}$ the fiber of $V$ at the generic point $\eta$ of $\bbP^1$. Observe that $V_{\eta}$ is a 
zero cycle in $\mathfrak{X}$ of degree $(\alpha\cdot \xi)$.
Since $\iota(\gamma)$ is an element of $\Aut^t(X)$, 
it fixes every fiber $X_t$ of $\pi$ and permutes its irreducible components. We can therefore assume  $V_\eta\neq \emptyset$;
then $\iota(\gamma)_*V$ is the Zariski closure of  $(\iota(\gamma))|_{\mathfrak{X}})_*V_{\eta}$.

Pick a divisor $D$ on $X$ which represents $\gamma$ and denote by $D_{\eta}$ its generic fiber.
Then $\Tr(\gamma)$ is represented by $D_{\eta}$.
On $X_\eta$,  $\iota(\gamma)|_{X_{\eta}}$  is the translation by $\Tr(\gamma)$; thus,  
$(\iota(\gamma)|_{X_\eta})_*V_{\eta}$ and $V_{\eta}+(\alpha\cdot \xi)D_{\eta}$ are linearly equivalent as divisors of $\Pic(X_\eta)$, and of $\Pic(\mathfrak{X})$. In other words, there exists a rational function 
$h$ in $\bfk(\bbP^1)(\mathfrak{X})=\bfk(X)$ such that 
\begin{equation}
{\mathrm{Div}}(h)|_{\mathfrak{X}}=(V_{\eta}+(\alpha\cdot \xi)D_{\eta})-(\iota(\gamma)|_{\mathfrak{X}})_*V_{\eta},
\end{equation}
where ${\mathrm{Div}}(h)$ is the divisor  associated to $h$.
Taking Zariski closures in $X$, we get 
\begin{equation}
{\mathrm{Div}}(h)=(V+(\alpha\cdot \xi)\gamma)-(\iota(\gamma)|_{\mathfrak{X}})_*V+W
\end{equation}
where $W$ is a vertical divisor i.e. $W\cap \mathfrak{X}=\emptyset$. It follows that 
$\iota(\gamma)_*\alpha=\alpha+(\alpha\cdot \xi)\gamma \mod\,\, R(X),$ which concludes the proof.
\end{proof}

\subsubsection{Action of $\xi^\perp$ on $\Axis$}

Denote by $\overline{\iota}\colon \xi^\perp/R(X)\to \Aut(\NS(X)/R(X))$ the induced action. 

\begin{lem}\label{lem:eqconj}
The homomorphism $\overline{\iota}$ is injective. 
For every $f\in \Aut(X)$ and $\gamma$ in $\xi^\perp$ $($resp. in $\xi^\perp/R(X))$, we have 
$
\iota(f^*\gamma)=f^{-1}\circ \iota(\gamma) \circ f.
$
\end{lem}

\begin{proof}
Let $\gamma$ be in the kernel of $\overline{\iota}$. Apply Property~(3) of Proposition~\ref{pro:T} to the class
$\alpha = \bfe_0-\bfe_1-\bfe_2$; since $\alpha\cdot \xi =1$, we obtain $\gamma=0 \mod\,\, R(X)$. This proves that
$\overline{\iota}$ is injective. If $f$ is an element of $\Aut(X)$, then $f^*$ fixes $\xi$ and it stabilizes $\xi^\perp$ and $R(X)$. 
Denote by ${\overline{f^*}}$  the action of $f$ on $\NS(X)/R(X)$.
For all $\gamma \in \xi^\perp/R(X)$ and $\alpha\in \NS(X)/R(X)$, we obtain the following equalities in $\NS(X)/R(X)$
\begin{align}
{\overline{\iota}}(f^*\gamma)(\alpha) & =    \alpha + (\alpha \cdot \xi) f^*\gamma \\
& = {\overline{f^*}}\left( \,     ({\overline{f^*}})^{-1}\alpha + (({\overline{f^*}})^{-1}\alpha \cdot \xi)\,  \gamma   \,  \right) \\
& =   ({\overline{f^*}}\circ {\overline{\iota}}(\gamma)) (({\overline{f^*}})^{-1}\alpha)
\end{align}
because $f^*$ preserves   intersections and fixes $\xi$.
Thus,  ${\overline{\iota}}(f^*\gamma)= {\overline{f^*}}\circ {\overline{\iota}}(\gamma) \circ ({\overline{f^*}})^{-1} $. Then 
$\iota(f^*\gamma)=f^{-1}\circ \iota(\gamma) \circ f$, $\Aut^t(X)$ is a normal subgroup, and  $\overline{\iota}$ is injective. \end{proof}

Denote by $q^P$ the quotient map $ P\to \xi^\perp/R(X)_{\R}$, and by  $q^A\colon A\to \xi^\perp/R(X)_{\R}$ its restriction
to $A=R(X)^\perp/\langle\xi\rangle_\R$.
By the Hodge index theorem, 
\begin{equation}\label{eqnorxrxp}
R(X)_{\R}\cap R(X)^{\perp}_{\R}=\langle \xi\rangle_{\R}.
\end{equation}
So $q^A$ is an isomorphism. For $\gamma\in \xi^\perp$, denote by $a(\gamma)$ the unique
element of $A$ such that $q^A(a(\gamma))= q^P(q_P(\gamma))$. Alternatively  $a(\gamma)$ is the orthogonal 
projection of $q_P(\gamma)$ onto $A\subset P$. The image of the linear map $a$ is $A(\Z)$, a lattice in $A$.

\smallskip

We want to understand the action of $\xi^\perp$ on $\Axis$ given by $\gamma\mapsto \iota(\gamma)_\Axis$.
Set $\Axis_0=q_{Q}^{-1}(\Axis)\subset Q_0$. Since every element $\alpha$ of $\Axis_0$
satisfies $\alpha\cdot \xi=3$, the third property of Proposition~\ref{pro:T} gives 
\begin{equation}
\iota(\gamma)_*\alpha\in (\alpha+3\gamma+R(X)_{\R})\cap \Axis_0 \quad (\forall \alpha \in \Axis_0, \; \forall \gamma\in \xi^\perp);
\end{equation}
since $\Axis_0= \alpha+R(X)^{\perp}_{\R}$, we get 
\begin{equation}\label{eq:TTralpha}
\iota(\gamma)_*\alpha-\alpha\in (3\gamma+R(X)_{\R})\cap R(X)^{\perp}_{\R} \quad (\forall \alpha \in \Axis_0, \;\forall \gamma\in \xi^\perp).
\end{equation}

\begin{lem}\label{lem:faithful} $ \, $ 
\begin{enumerate}
\item  for $\gamma\in \xi^\perp$, $\iota(\gamma)_\Axis\colon \Axis\to \Axis$ is  the translation 
$v\mapsto v+3 a(\gamma)$;
\item $\iota(\xi^\perp)_\Axis$ is
 a discrete and co-compact group of translations of $\Axis$;
\item this group $\iota(\xi^\perp)_\Axis$ depends only on $\Irr(X)\subset \xi^\perp$ $($it does not depend on other geometric features of $X.)$
\end{enumerate}
\end{lem}

\begin{proof}
Pick any $\alpha$ in $\Axis_0$. The Equation~\eqref{eq:TTralpha} gives $q^A\left( q_P(\iota(\gamma)_*\alpha-\alpha) \right) =3q^P(q_P(\gamma))$.
Thus,  $\iota(\gamma)_*\alpha-\alpha=3 a(\gamma)$ modulo $\Z \xi$. The three properties follow from this equation.  \end{proof}

\subsection{Automorphisms and conjugation}\label{par:Thm_Halphen}

On $Q$, $\Aut(X)$ acts by affine isometries for the (euclidean) metric induced by the
intersection form. In $Q$, we have
\begin{itemize}
\item the axis $\Axis$, an affine subspace of dimension $\dim R(X)^\perp$;
\item the convex cylinder $Q^+$: it is equal to the pre-image of the compact polytope $D(X)\subset Q/\Axis$ by $q_{Q/A}$.
The diameter of $D(X)$ is bounded by some constant $r>0$ that depends only on $m$ (not on the geometry of $X$, see Equation~\eqref{ieqpd}).
\item the projection $\bfe$ of $\bfe_0$, an element of $Q^+$; we shall denote by $\bfo$ the orthogonal projection of 
$\bfe$ on $\Axis$ (the projection of $\bfo$ in $D(X)$ coincides with the center of mass of $D(X)$, see Section~\ref{par:cylinderQ+}).
\end{itemize}
The group $\Aut(X)$ preserves $Q^+$ and $\Axis$. In what follows, we view $Q$ as a vector space, whose origin 
is $\bfo$; then, every element $f$ of $\Aut(X)$ gives rise to an affine isometry 
\begin{equation}
f_Q(v)=L_f(v)+t(f)
\end{equation}
where $L_f$ is the linear part of $f$ (fixing $\bfo$) and $t(f)$ is the translation part. 
Note that $t(f)=f_Q(\bfo)$ is an element of the lattice $\Axis\cap Q(\Z)\subset \Axis$. 
Since $f_Q$ is an isometry and preserves the lattice $Q(\Z)$, we have $L_f\in O(Q(\Z)):=O(Q)\cap \GL(Q(\Z))$, where $O(Q)$ is the orthogonal group of $(Q,\|\cdot\|_Q)$ and $\GL(Q(\Z))$ is the subgroup of $\GL(Q)$ that preserves $Q(\Z)$. Since $O(Q(\Z))$ is a finite subgroup of $\GL(Q(\Z))$ and the rank of $Q(\Z)$ is $9$, independently of $m$ and $X$, we get the following lemma.

\begin{lem}\label{lemlinearpartbound}
The group of all linear parts $\{L_f\; ; \; f\in \Aut(X)\}$ is finite: its cardinality is uniformly bounded, by an 
integer  that does not depend on $X$. 
\end{lem}

Moreover, denote by  $O(Q(\Z) ; \Axis)$ the subgroup of $O(Q(\Z))$  that preserves $\Axis$ and $O(Q(\Z); \Axis)|_{\Axis}\subseteq O(\Axis)$ the image of $O(Q(\Z); \Axis)$ under the restriction to $\Axis.$ Denote by $f_{\Axis}$ the restriction of $f$ to $\Axis$. We have $f_{\Axis}(v)=L_f|_{\Axis}(v)+t(f)$ where $L_f|_{\Axis}$ is the restriction of $L_f$ to $\Axis$.  
Then $L_f|_{\Axis}\in O(Q(\Z); \Axis)|_{\Axis}.$

\begin{lem}\label{lem:deg-estimate}
There is a  constant $B_2(m)\geq 1$ such that 
\[
-B_2(m) (1+\parallel t(f) \parallel_Q) \leq \deg(f) - \frac{1}{2}\parallel t(f)\parallel_Q^2\leq B_2(m) (1+ \parallel t(f)\parallel_Q )
\]
for all $f\in \Aut(X)$. In particular, there is a constant $B_3(m)\geq 1$ such that 
\[
B_3(m)^{-1}\deg(f)  \leq (\parallel t(f) \parallel_Q+1)^2\leq  B_3(m) \deg(f).
\]
\end{lem}

\begin{proof}
The degree $\deg(f)$ is equal to the intersection number $\bfe_0\cdot f^*\bfe_0$. Since $f^*\bfe = L_f(\bfe)+t(f)$, 
 Equation~\eqref{eqdv} gives\begin{equation}
\deg(f)=1+\frac{1}{2}\parallel (L_f-\Id_Q)(\bfe)+t(f)\parallel_Q^2.
\end{equation}
The first estimate follows because $L_f$ is an isometry, and the length of $\bfe$ (i.e. the length $\parallel \bfe-\bfo\parallel_Q$ in $Q$) 
is bounded by a constant that depends only on $m$. To derive the second inequalities from the first, we set 
\begin{equation}C_1(m) =\sup_{x\geq 0}\frac{(x+1)^2}{\max\{1, \frac{1}{2}x^2-B_2(m)(1+x)\}}\geq 1.
\end{equation}
Then
\begin{align}
C_1(m)\deg(f) &\geq C_1(m)\max\{1, \frac{1}{2}\parallel t(f)\parallel_Q^2-B_2(m) (1+\parallel t(f) \parallel_Q)\}\\
&\geq (\parallel t(f) \parallel_Q+1)^2.
\end{align}
And with $C_2(m)=\max\{B_2(m),1/2\}$, we get 
\begin{equation}
\deg(f) \leq  \frac{1}{2}\parallel t(f)\parallel_Q^2+ B_2(m) (1+ \parallel t(f)\parallel_Q )\leq C_2(m)(1+\parallel t(f) \parallel_Q)^2.
\end{equation}
Setting $B_3(m):=\max\{C_1(m),C_2(m)\}$, we conclude the proof.
\end{proof}

The group $\Aut^t(X)$ is a finitely generated abelian group. Denote by $\Aut^t_{tor}(X)$ its torsion subgroup and set $\Aut^t_{fr}(X):=\Aut^t(X)/\Aut^t_{tor}(X).$
Then $\Aut^t_{fr}(X)$ is isomorphic to $\Z^d$ for $d=\dim_\R (\xi^\perp_\R/R(X)_\R)$.

\begin{eg} Consider the elliptic curve $E=\C/\Z[\jj]$, where $\jj$ is a primitive cubic root of unity, and then the abelian 
surface $A=E\times E$. The group $\GL_2(\Z[\jj])$ induces a group of automorphisms of $A$; let $\eta$ be the 
automorphism defined by $\eta(x,y)=(\jj x, \jj y)$, and let $G\simeq \Z/3\Z$ be the subgroup of $\Aut(A)$
generated by $\eta$. Let $X$ be a minimal resolution of $X_0=A/G$. Then $X$ is a smooth rational surface. 
Since $G$ is contained in the center of $\GL_2(\Z[\jj])$, the group $\GL_2(\Z[\jj])$ acts on $X_0$ and on $X$
by automorphisms. The singularities of $X_0$ correspond to the points of $A$ which are fixed by some non-trivial element
in $G$; the fixed points of $G$ form a subgroup $K$ of $A$ and the action of $K$ on $A/G$ by translations 
induces an action on $X$ by automorphisms of finite order. Now, consider the fibration $\pi\colon X\to \bbP^1_\bfk$
induced by the first projection $\pi_1\colon A\to E$. This is an example of a (blow up of a) Halphen fibration. With respect to 
this fibration, elements of $K$ of type $(x,y)\mapsto (x,y+s)$ determine elements of $\Aut^t_{tor}(X)$ and 
elements of $\GL_2(\Z[\jj])$ acting by $(x,y)\mapsto (x, y + a x)$ determine elements of $\Aut^t(X)$ of infinite order.
\end{eg}

\begin{rem}\label{rem:torsion_in_Autt}
Let $\rho_3\colon \GL(\NS(X;\Z))\to \GL(\NS(X;\Z/3\Z))$ be the homomorphism given by reduction modulo $3$. 
Its kernel is a torsion free subgroup of index $\leq  \vert \GL_{10}(\Z/3\Z)\vert=(3^{10}-1)(3^{10}-3)\cdots (3^{10}-3^9)$.
So, if $f\in \Aut(X)$ is elliptic, $(f^*)^k=\Id$ for some $k\leq 3^{100}$. Now, assume that $f\in \Aut^t(X)$ has finite order,
choose such a power $k$, and pick a multisection $E$ of $\pi$ with negative self-intersection; such a curve exists, for
instance among the $9$ exceptional divisors blown-down by $\epsilon\colon X\to \bbP^2_\bfk$, and we may even assume
that $E$ intersects a general fiber of $\pi$ in at most $m$ points. Since $E$ is determined
by its class $[E]$, $f^k$ fixes $E$ and $f^{km!}$ fixes each of the points given by $E$ in the general fibers of $\pi$. 
Since a general fiber is an elliptic curve in which $f^{km!}$ has a fixed point, we conclude that the order of $f$ is at most $24\times 3^{100} \times m!$.
Of course, a more detailed analysis would give a much better uniform upper bound.
\end{rem}

By Lemma \ref{lemlinearpartbound},  there is a constant $N\geq 1$ that does not depend on $X$, such that for every $h\in \Aut^t(X)$, $L_{h}^{N}=\Id.$
Set 
\begin{equation}
\Aut^t_N(X):=\iota(N\xi^{\perp})\subset \Aut^t(X).
\end{equation} 
The group $\Aut^t_N(X)$ has finite index in $\Aut^t(X)$.
By Lemma~\ref{lem:eqconj}, the action of $\Aut^t(X)$ on $\NS(X)$ is faithful, hence so is its action on $Q$. This implies that $\Aut^t_N(X)$ is torsion free, because $\Aut^t_N(X)$ acts by translation on $Q$, and a translation of finite order is the identity.
So $\Aut^t_N(X)$ is isomorphic to $\Z^d$ for $d=\dim_\R (\xi^\perp_\R/R(X)_\R)$. Since $\Aut^t(X)$ is a normal subgroup of $\Aut(X)$, $\Aut^t_N(X)=\{h^N\; ; \; h\in \Aut^t(X)\}$ is also normal in $\Aut(X)$.

\begin{lem}\label{lembod} There is a constant $B_4(m)\geq 1$ such that, for every $f\in \Aut(X)$, there exists $h\in \Aut^t_N(X)$ with 
$L_h=\Id$ and $\deg(h\circ f)\leq B_4(m)$.
\end{lem}

\begin{proof}The action of $\Aut^t(X)=\iota(\xi^\perp)$ on $\Axis$  is co-compact. Since $\Aut^t_N(X)$ is of finite index in $\Aut^t(X)$, the action of $\Aut^t_N(X)$ is also co-compact. So the Voronoi cell
$\{ u\in \Axis\, ; \; \dist_Q(u,\bfo)\leq \dist(u,\iota(\gamma)_Q\bfo), \; \forall \iota(\gamma) \in \Aut^t_N(X)\}$ 
is a compact fundamental domain; let $r(m)$ be the radius of this domain.
There is an $h$ in $\Aut^t_N(X)$ such that  $t(h\circ f)=h(f(\bfo))$ has length at most $r(m)$. The conclusion follows from Lemma~\ref{lem:deg-estimate}.
\end{proof}

\begin{thm}\label{thm:Conjugacy-in-Aut-Halphen}
Let $X$ be a Halphen surface of index $m$.
Let $f$ and $g$ be two automorphisms of $X$.
If $f$ and $g$ are in the same conjugacy class of $\Aut(X)$, there exists an element 
$h$ of $\Aut(X)$ such that $$h\circ f \circ h^{-1}=g \; \, {\text{ and }} \; \, 
\deg(h):=(\bfe_0\cdot h^*\bfe_0)\leq A(m)(\deg(f)+\deg(g)),$$ for some constant $A(m)$ that depends only
on $m$.
\end{thm}

\begin{proof}
The group $\Aut(X)$ acts by conjugacy on $\Aut^t(X)$ and on $\Aut^t_N(X)$; if $g$ is in $\Aut(X)$, we denote by $A_g$ the 
endomorphism $t\mapsto g\circ t \circ g^{-1}$ of $\Aut^t(X)$. 

Start with a conjugacy $c_0\circ f \circ c_0^{-1}=g$, for some $c_0\in \Aut(X)$. According to Lemma~\ref{lembod}
there is an element $s_0\in \Aut^t_N(X)$ such that $c_0=s_0\circ c_1$ with $\deg(c_1)\leq B_4(m)$.  
Then $c_1\circ f \circ c_1^{-1}$ and $g$ have the same class in the quotient $\Aut(X)/\Aut^t_N(X)$, so that $c_1\circ f \circ c_1^{-1}=s'\circ g$ 
for some $s'\in \Aut^t_N(X)$ (indeed for $s'=[s_0^{-1},g]$). 

Let us use additive notations for $\Aut^t(X)$. The conjugacy equation $h\circ f \circ h^{-1}=g$, with $h=s\circ c_1$, is now equivalent to 
$s'= (A_g-\Id)(s)$, and we know that there is at least one solution $s_0$. 

\begin{lem}\label{lemineqonautt}There is $s\in \Aut^t_N(X)$ such that $s'= (A_g-\Id)(s)$ and 
\[
\|t(s)\|_Q\leq B_5(m)\|t(s')\|_Q
\]
for some constant $B_5(m)$ that depends only
on $m$.
\end{lem}
Note that the translation part $s'_Q(\bfo)= t(s')=t(c_1\circ f \circ c_1^{-1}\circ g^{-1})$ satisfies
\begin{equation}
\parallel  t(s')\parallel_Q\leq A_1(m)(\parallel t(g)\parallel_Q + \parallel t(f)\parallel_Q )+A_2(m)
\end{equation} 
for some constants $A_1(m)$ and $A_2(m)$ that depend only on $m$; one can take $A_1(m)=M^4$ with 
$M=\max \{\parallel L_h \parallel \; ; \; h\in \Aut(X)\}$ and $A_2(m)=2 A_1(m)\parallel t(c_1)\parallel_Q$. 
Then by Lemma \ref{lemineqonautt},
we have 
$\|t(s)\|_Q+1\leq A_3(m)(\|t(f)\|_Q+\|t(g)\|_Q+2)$ where $A_3(m)=B_5(m)(A_1(m)+A_2(m))+1.$
By Lemma~\ref{lem:deg-estimate}, 
\begin{align}
\deg(h) &\leq \deg(c_1)\deg(s) \\
&\leq B_4(m)\deg(s)\\
&\leq B_4(m)A_3(m)^2((\|t(f)\|_Q+1)+(\|t(g')\|_Q+1))^2 \\ 
&\leq 2B_4(m)A_3(m)B_3(m)(\deg(f)+\deg(g)),
\end{align}
which concludes the proof.
\end{proof}

\proof[Proof of Lemma \ref{lemineqonautt}]
Lemma~\ref{lem:faithful} shows that  $\Aut^t_N(X)$ acts on $\Axis$ faithfully
by translations. 
We may identify $\Aut^t_N(X)$ to 
the orbit of $\bfo$ in $\Axis$ via the map $h\in \Aut^t_N(X)\mapsto h_Q(\bfo)=t(h)$. The automorphism $g$ acts also
on $\Axis$, by $g_Q(v)=L_g(v)+t(g)$, and the equation $(A_g-\Id)(s)=s'$ corresponds to the equivalent
equation 
\begin{equation}\label{eq:Lg-sQ}
(L_g|_{\Axis}-\Id)t(s)=t(s').
\end{equation} 

If $E=(e_1,\dots,e_d)$ is a basis of $\Aut^t_N(X)$, denote by $\|\cdot\|_E$ the euclidean norm 
\begin{equation}
\parallel \sum_{i=1}^{d}a_ie_i\parallel_E =(\sum_{i=1}^da_i^2)^{1/2}.
\end{equation}
There is a constant $C(E)\geq 1$ such that for every $s\in \Aut^t_N(X)$, we have 
\begin{equation}
C(E)^{-1}\|t(s)\|_Q \; \leq \;  \|s\|_E \; \leq \; C(E)\|t(s)\|_Q.
\end{equation}

For every endomorphism $\Phi: \Aut^t_N(X)\to \Aut^t_N(X)$, there are two basis $E_1(\Phi)=(u_1,\dots,u_d)$, $E_2(\Phi)=(v_1,\dots,v_d)$ of $\Aut^t_N(X)$,
a sequence of integers $(a_1, \ldots, a_d)\in \Z^d$ and an index $\ell\leq d$ such that 
\begin{itemize}
\item $\Phi(u_i)=a_iv_i$  for all $1\leq i \leq d$;
\item $a_i=0$ if and only if  $i > \ell$.  
\end{itemize}
Then for every $v\in \Phi(\Aut^t_N(X))$, we may write $v=\sum_{i=1}^\ell a_ix_iv_i$, where $x_i\in \Z$; equivalently, $v=\Phi(\sum_{i=1}^\ell x_iu_i)$.  
This shows that for every $v\in \Phi(\Aut^t_N(X))$, there is $u\in \Phi^{-1}(v)$ such that $\|u\|_{E_1(\Phi)}\leq \|v\|_{E_2(\Phi)}.$
Setting $C(\Phi):=C(E_1)C(E_2)$, every $v\in \Phi(\Aut^t_N(X))$ is the image of an element $u\in \Aut^t_N(X)$ such that 
\begin{equation}\label{equpreimagebound}\|t(u)\|_Q\leq C(\Phi)\|t(v)\|_Q.
\end{equation}
We apply this remark to $\Phi_g:=A_g-\Id$. For every $m$, there are only finitely many possibilities for $\Axis$ and $L_g$, so there are only finitely many possible $L_g|_\Axis\in O(Q(\Z); \Axis)|_{\Axis}$. Thus, there are only finitely many possible values for the constant $C(\Phi_g)$, and the proof follows from Inequality \ref{equpreimagebound}.
\endproof

%
%
\section{Invariant pencils, conjugacy classes, and limits}\label{par:pen_con_lim}
%
%

\subsection{Halphen and Jonqui\`eres twists}\label{par:Halphen_Jonquieres_basics}
  
\subsubsection{Halphen twists}\label{par:Halphen_twists_basics}

A birational transformation $f$ of the plane is a Halphen twist if and only if $(\deg(f^m))_{m\geq 0}$ grows quadratically,
like $c(f) m^2$ for some constant $c(f)>0$ (see~\cite{Cantat:Annals, Boucksom-Favre-Jonsson:Duke, Blanc-Deserti:2015}). 
It is a Jonqui\`eres twist if and only if $(\deg(f^m))_{m\geq 0}$ grows like $c(f) m$ for some $c(f)>0$. And it is elliptic if $(\deg(f^m))_{m\geq 0}$
is bounded. If $f$ is not a twist or an elliptic element, then its dynamical degree 
\begin{equation}
  \lambda(f):=\lim_{m\to +\infty}\left( \deg(f^m)^{1/m}\right)
\end{equation}
is larger than $1$ and, in fact, larger than the Lehmer number  (see \cite{Diller-Favre}, and \cite{Cantat:Annals, Blanc-Cantat}). 
Halphen twists have been studied by Gizatullin in~\cite{Gizatullin:1980}. He proved the following result, for which we also refer to~\cite[Theorem~4.6]{Cantat-Guirardel-Lonjou}, 

\begin{thm}\label{thm:Gizatullin}
There is an integer $B_H \leq 8!$ with the following property. Let $\bfk$
be an algebraically closed field. Let $f\in \Bir(\bbP^2_\bfk)$ be a Halphen twist. Then, up
to conjugacy, $f$ preserves a Halphen pencil $\HP$ $($of some index $m\geq 1)$,
and:
\begin{itemize}
\item[(H.1)] this Halphen pencil provides an elliptic fibration of the corresponding
Halphen surface $X$,
\item[(H.2)]  the kernel of the homomorphism $\Aut(X)\to \GL(\NS(X))$ is finite, of order $\leq B_H$,
\item[(H.3)]  $f$ does not preserve any other pencil of curves.
\end{itemize}
\end{thm}
``Up to conjugacy'' means that there is a birational map $\alpha\colon \bbP^2_\bfk\dasharrow \bbP^2_\bfk$ such that $\alpha^{-1}\circ f \circ \alpha$
preserves such a Halphen pencil $\HP$; of course, $f$ preserves the pencil $\alpha(\HP)$, and Property (H.3) means that there is no other
invariant pencil.  The first item means that the fibration $\pi\colon X\to \bbP^1_\bfk$ induced by $\HP$ is not quasi-elliptic (see Section~\ref{par:quasi-elliptic} above and~\cite{Cantat-Guirardel-Lonjou}). Of course, $\HP$, its index $m$, and the conjugacy $\alpha$ depend on $f$. 

Fix such a Halphen twist $f$, and after conjugacy, assume that $f$ preserves a Halphen pencil $\HP$, as
in Theorem~\ref{thm:Gizatullin}. We can now apply the results of Section~\ref{par:Halphen-Surfaces} to $\HP$. 
First, the group $\Bir(\bbP^2;\HP)$ is conjugate to $\Aut(X)$ by the birational morphism $\epsilon\colon X\to \bbP^2_\bfk$; in particular
$f$ is conjugate to an automorphism of $X$. Second, Lemma~\ref{lem:deg-estimate} and Property (H.2) imply
that
\begin{itemize}
\item[(H.4)] The group of vertical translations $\Aut^t(X)$ is a normal subgroup of $\Aut(X)$ of index $\leq I(m)$ for some positive
integer $I(m)$ that depends only on $m$. Some positive iterate of $f$ is contained in this group.
\end{itemize}
The fifth property we shall need is the existence of an integer $J(m)$ (depending only on $m$) such that:
\begin{itemize}
\item[(H.5)] If $h$ is an element of $\Aut(X)$  of
finite order, then its order divides $J(m)$. If $h\in \Aut(X)$ is  elliptic, then it has finite order (and thus $h^{J(m)}=\Id$).
\end{itemize}
To prove this last fact, we first apply (H.4): $h^{I(m)}$ is a vertical translation of $X$. 
From Section~\ref{par:vertical-translations},  the order of the torsion part of $\Pic^0(\mathfrak{X})(\bfk(\bbP^1_\bfk))$, hence of 
$\Aut^t(X)$, 
is  bounded by some constant $J_0(m)$, because there are only finitely many possibilities for $\xi^\perp/R(X)$. Thus if $h$ has finite
order, its order divides $J(m)=I(m)\times J_0(m)$. 
If $h$ is elliptic, $h^*\in \GL(\NS(X))$ is an element of finite order. By (H.2), $h$ itself has finite order and we conclude with the first assertion.

\subsubsection{Jonqui\`eres twists}\label{par:Jonq-twist-basics}  Similarly,
\begin{itemize}
\item[(J.1)] a Jonqui\`eres twist preserves a unique pencil of curves, and it is 
a pencil of rational curves;
\item[(J.2)] after conjugacy, we can take this pencil to be given by the lines through the point
$[0:1:0]\in \bbP^2_\bfk$.
\end{itemize} 
With coordinates $(x,y)=[x:y:1]$ on the affine part $z\neq 0$,  any birational transformation $f$
preserving this pencil can be written $f(x,y)=(A(x),B_x(y))$ where $A\in \PGL_2(\bfk)$ and $B_x\in \PGL_2(\bfk(x))$
(note that such a transformation can be either elliptic or a Jonqui\`eres twist).

\subsubsection{The set $Pog(d)$}  

Let $\bfk$ be an algebraically closed field. For every integer $d\geq 1$, denote by $\Bir_d(\bbP^2_\bfk)$ the algebraic 
variety of birational transformations of $\bbP^2_\bfk$ of degree $d$. Recall Theorem~\ref{thm:Xie}: for every $m\geq 0$, the function 
\begin{equation}
f\in \Bir_d(\bbP^2_\bfk) \mapsto \deg(f^m)\in \Z_+
\end{equation}
is lower semi-continous for the Zariski topology; in other words, the sets 
$\{g\in \Bir_d(\bbP^2_\bfk)\; ; \;  \deg(g^m)\leq D\} $ are Zariski closed: degrees of iterates can only decrease under specialization (see~\cite[Lemma~4.1]{Xie:Duke}). The dynamical degree
$f\mapsto \lambda(f)$ is also lower semi-continous (see \cite[Theorem~4.3]{Xie:Duke}). In particular, the set $Pog(d) \subset \Bir_d(\bbP^2_\bfk)$ defined by 
\begin{equation}
Pog(d)=\{ g \in \Bir_d(\bbP^2_\bfk)\; ;\;  \lambda(g)=1\}
\end{equation}
is a Zariski closed subset of $\Bir_d(\bbP^2_\bfk)$; here, $Pog$ stands for ``polynomial growth'', because an element of $Pog$ 
has bounded, linear, or quadratic degree growth: $Pog(d)$ is the disjoint union of the three subsets 
\begin{align}
\label{eq:hald} Hal(d) & =  \{ g \in \Bir_d(\bbP^2_\bfk)\; ;\; g \;  {\text{ is a Halphen twist}} \}, \\
Jon(d) & = \{ g \in \Bir_d(\bbP^2_\bfk)\; ;\; g \;  {\text{ is a Jonqui\`eres twist}} \}, \\
\label{eq:elld} Ell(d) & =  \{ g \in \Bir_d(\bbP^2_\bfk)\; ;\; g \;  {\text{ is elliptic}} \}.
\end{align}
If $X$ is a Zariski closed subset of $\Bir_d(\bbP^2_\bfk)$, we denote by $Hal^X(d)$, $Jon^X(d)$, and $Ell^X(d)$
the respective intersections of $X$ with these three sets. 

\subsection{Invariant pencils of Halphen twists}\label{par:DegBoundHalphen}
 
\begin{thm}\label{thm:bounds for halphen}   
Let $W$ be an irreducible subvariety of $\Bir_d(\bbP^2_\bfk)$. If $Hal^W(d)$ is Zariski dense in $W$, then 
 there is an integer $d_W$
  and a proper Zariski closed subset $Z_W\subset W$ such that every element $g\in W\setminus Z_W$ is a Halphen twist and
\begin{enumerate}
\item  the unique $g$-invariant pencil has degree $\leq d_W$;
\item there is a birational map $\alpha_g\colon \bbP^2_\bfk\dasharrow \bbP^2_\bfk$ of degree
$\leq d_W$  that maps this pencil to a Halphen pencil of degree at most $d_W$.
\end{enumerate}
\end{thm} 

\begin{proof}  
First, we view the generic element of $W$ as a birational transformation $f$ of $\bbP^2$ defined over the 
function field $K=\bfk(W)$; for every $t\in W$, the specialization 
$f_t$ of $f$ is a birational map of $\bbP^2_\bfk$ of degree $d$. 

\begin{lem}\label{lem:Halphen-f}
The birational transformation $f$ of $\bbP^2_K$ is a Halphen twist.
\end{lem}

\begin{proof}[Proof of Lemma~\ref{lem:Halphen-f}]
If $f\in \Bir(\bbP^2_K)$ were loxodromic, with first dynamical degree $\lambda$, then by the semi-continuity Theorem~4.3 of \cite{Xie:Duke}
there would exist a Zariski dense open subset $W'\subset W$ such that for every $t\in W'$, $f_t$ would be loxodromic with $\lambda(f_t)> (\lambda+1)/2>1$, and 
this would contradict our assumption.

So, $f$ is either elliptic, or a Jonqui\`eres twist, or a Halphen twist. If it 
were elliptic, the degrees of its iterates would be bounded, and by semi-continuity of the degrees every specialization 
$f_t$ would be elliptic, contradicting the existence of Halphen twists $f_t$ in $W$. 

Suppose now
that $f$ is a Jonqui\`eres twist. 
By semi-continuity of the degrees, for every $t\in B, n\geq 0, \deg f_t^n\leq \deg f^n\leq Cn+D$ for some $C,D>0$, contradicting the existence of Halphen twists $f_t$ in $W$.
\end{proof}

Thus, $f\in \Bir(\bbP^2_K)$ is a Halphen twist. Replacing $K$ by some finite extension $K'$, 
there is a rational function $\pi\colon \bbP^2_{K'}\dasharrow \bbP^1_{K'}$ 
and a birational map $\alpha\colon \bbP^2_{K'}\dasharrow \bbP^2_{K'}$ such that $f$ 
preserves the pencil defined by $\pi$, and $\alpha$ maps this pencil to some
Halphen pencil. According to Property~(H.4), $\pi\circ f^{\ell}=\pi$
for some integer $\ell\geq 1$ (with $\ell \leq I(m)$, $m$ being the index of the $f$-invariant pencil). 
Write $K'=\bfk(W')$, for some finite base change $\tau\colon W'\to W$. 
Let $\deg(\pi)$, $\deg(\alpha)$, and $\delta$ be the degrees of $\pi$,   $\alpha$,  and the Halphen pencil respectively. For $s$ in the complement of a 
divisor $D'$ of $W'$, $\pi$ specializes to a rational function $\pi_s$ of degree $\deg(\pi)$ defining a pencil of genus $1$, and
the specialization $\alpha_s$ is a birational transformation of $\bbP^2_\bfk$ of degree $\deg(\alpha)$ that maps the $f_{\tau(s)}$-invariant pencil to 
a Halphen pencil of degree $\delta$.
Thus, for $t\in W\setminus \tau(D')$, the birational transformation $f_t$ preserves a pencil of degree $\deg(\pi)$ by curves
of genus $1$. Since Jonqui\`eres twists preserve a unique pencil, made of rational curves (see~\S~\ref{par:Jonq-twist-basics}), $f_t$ can not be such a twist. 
If $f_t$ were elliptic, $f_t^k=\Id$ for some integer $k\leq 24 \times 3^{100} \times m!$ (see Remark~\ref{rem:torsion_in_Autt}). But the subset of birational maps $g\in W$ of order $\leq 24 \times 3^{100} \times m! $ is a Zariski closed subset, and
is a proper subset because $W$ contains Halphen twists. Thus, for $t$ in the complement of a proper Zariski closed subset
of $W$, $f_t$ is a Halphen twist preserving a pencil of degree $ d_W:=\deg(\pi)$. 
\end{proof}

Note that Theorem~\ref{thm:bounds for halphen} can also be obtained by looking at the action 
of Halphen twists on the Picard-Manin space, with the help of hyperbolic geometry (instead of algebraic 
geometry as done in the previous proof). This is presented in the appendix to this paper, on the arXiv version of it
(see~\cite{CDX:LongArxiv}). 


\begin{cor}\label{cor:HT-degree-bounds}
The set $\Hal(d)\subset \Bir_d(\bbP^2_\bfk)$ is constructible in the Zariski topo\-logy, and
one can find a finite stratification of $\Hal(d)$ such that, on each stratum, the degree of the 
$f$-invariant pencil is constant and there is a birational map $\alpha_f$
of fixed degree mapping this pencil to  some Halphen pencil $\HP_f$ of fixed degree.\end{cor}

Note that $\HP_f$ and $\alpha_f$ depend on $f$ and are not canonically defined; for instance $\alpha_f$
can be composed with a birational transformation preserving $\HP_f$, and $\HP_f$ can be changed into
a projectively equivalent Halphen pencil.

\begin{proof}
Set $V_0=Pog(d)$, and consider the Zariski closure $\overline{\Hal(d)}$ of $\Hal(d)$ in $V_0$. 
Let $(W_i)_{i=1}^{k}$ denote the irreducible components of maximal dimension of $\overline{\Hal(d)}$. 
For each $i\in  \{1, \ldots, k\}$, Theorem~\ref{thm:bounds for halphen} provides an effective divisor $D_i\subset W_i$ and a positive integer $d_i$ 
such that all elements of $W_i\setminus D_i$ are Halphen twists preserving a pencil of degree $d_i$. 
Let $V_1$ be the union of all $D_i$, together 
with all the irreducible components of $\overline{\Hal(d)}$ of dimension $<\dim(W_i)$. The same argument applies to the closure of 
$\Hal^{V_1}(d)$ in $V_1$. Thus, in at most $\dim \overline{\Hal(d)}$ steps, we obtain: $\Hal(d)$ is constructible and
there is a stratification of $\Hal(d)$ in locally closed sets on which the degree of the invariant pencil is piecewise
constant. 
The same argument, together with the construction of $\alpha$ given in the proof of Theorem~\ref{thm:bounds for halphen}, 
gives the last assertion of the corollary. \end{proof}

\subsection{Conjugacy classes of Halphen twists} 
 
\begin{thm}\label{thm:Halphen-Conjugacy-Classes}
For every integer $d\geq 1$, there is a positive integer ${\mathrm{HalCo}}(d)$ that satisfies the following property.
Let $f$ and $g$ be Halphen twists of degree $\leq d$. If $f$ is conjugate to $g$ in $\Bir(\bbP^2_\bfk)$, 
then there is a birational map $\psi$ of degree $\leq {\mathrm{HalCo}}(d)$ such that 
$\psi\circ f \circ \psi^{-1}=g$. 
\end{thm}

\begin{proof}
Since $f$ and $g$ are conjugate, their invariant pencils are birationally equi\-valent to the same Halphen pencil 
$\HP$, the index of which will be denoted by $m$. From Corollary~\ref{cor:HT-degree-bounds}, there exist an 
integer $hal(d)$ and two birational maps $\alpha$ and $\beta$ in $\Bir(\bbP^2_\bfk)$ such that 
\begin{itemize}
\item $\deg(\alpha)$, $\deg(\beta)$, and $\deg(\HP)=3m$ are bounded from above by $hal(d)$;
\item $f_0=\alpha\circ f\circ \alpha^{-1}$ and $g_0=\beta\circ g\circ \beta^{-1}$ are Halphen twists preserving $\HP$.
\end{itemize}

Note that $f_0$ and $g_0$ have degree at most $d \times hal(d)^2$ and are conjugate by a birational map of the plane
preserving the pencil $\HP$. Since $\Bir(\bbP^2;\HP)$ coincides with $\Aut(X;\pi)$, where $X$ is the Halphen surface obtained by blowing
up the $9$ base points of $\HP$, the birational transformations $f_0$ and $g_0$ can be seen as automorphisms
of $X$ which are conjugate by an element $h$ of $\Aut(X)$. Then, the conclusion follows from Theorem~\ref{thm:Conjugacy-in-Aut-Halphen}, 
with $HalCo(d)=2\times A(m)\times d \times hal(d)^2$.
\end{proof} 
  
 
\begin{cor}\label{cor:4.6}
The conjugacy class of a Halphen twist $f\in \Bir(\bbP^2_\bfk)$ is a constructible subset of $\Bir(\bbP^2_\bfk)$ $($endowed with its Zariski topology, as in \cite{Serre:Bourbaki-Cremona}, \S 1.6$)$. 
\end{cor} 

\begin{proof}
Let $d$ be the degree of $f$.
Theorem~\ref{thm:Halphen-Conjugacy-Classes} implies that the intersection of the conjugacy class of $f$ with 
$\Bir_{e}(\bbP^2_\bfk)$ is equal to 
\begin{equation*}
\Bir_{e}(\bbP^2_\bfk) \cap \left\{\psi\circ f\circ  \psi^{-1}\; ; \;  \psi\in \Bir(\bbP^2_\bfk) \; {\text{ and}} \;\deg(\psi)\leq  {\mathrm{HalCo}}(\max\{d,e\})\right\}
\end{equation*}
and this is a constructible set because the group law is  continuous for the Zariski topology (see \cite{Serre:Bourbaki-Cremona, Blanc-Furter:2013}).
\end{proof}

\subsection{Jonqui\`eres twists}

We now consider the Poincar\'e and conjugacy problems for Jonqui\`eres twists.


\begin{thm}\label{thm:bounds for jonquieres} 
Let $W$ be an irreducible subvariety of $\Bir_d(\bbP^2_\bfk)$. If $Jon^W(d)$ is dense in $W$, then every element
of $W$ is either elliptic or a Jonqui\`eres twist. Moreover, there is an integer $d_W\geq 1$ 
and a proper Zariski closed subset $Z_W\subset W$ such that if $g\in W\setminus Z_W$ is a Jonqui\`eres twist, then
\begin{enumerate}
\item  the unique $g$-invariant pencil has degree $\leq d_W$;
\item there is a birational map $\alpha_g\colon \bbP^2_\bfk\dasharrow \bbP^2_\bfk$ of degree
$\leq d_W$  that maps this pencil to the pencil of lines through a given point of $\bbP^2_\bfk$.
\end{enumerate}

\end{thm} 

The proof is the same as for Theorem~\ref{thm:bounds for halphen}, the main remark being that the generic element
$f$ of $W$, viewed as a birational transformation of $\bbP^2_K$ which is defined over the function field $K=\bfk(W)$,
is a Jonqui\`eres twist. Indeed, by lower semi-continuity of $\lambda_1$, $\lambda_1(f)=1$; by Theorem \ref{thm:bounds for halphen}, it cannot
be a Halphen twist; and if $f$ were elliptic, every element in $W$ would be elliptic by lower semi-continuity of degrees.

As a consequence, the degree of $f^n$ is bounded by $bn+ o(n) $ for some real number $b>0$. 
Then, to prove that no element of $W$ is a Halphen twist, assume by contradiction that $f_t$ is a Halphen twist for some parameter
$t\in W$. By definition, $\deg(f_t^n)\geq an^2-o(n^2)$ for some real number $a>0$. Taking $m$ large enough we get $\deg(f^m)<\deg(f_t^m)$, and this contradicts the semi-continuity of the degree map $g\mapsto \deg(g^m)$.

On the other hand, the following  examples show that  $\Jon(d)$ 
does not behave as well as $\Hal(d)$. These examples, and those in \S~\ref{par:limits_types}, are variations on examples given in~\cite{Blanc:Algebraic_Elements, Blanc-Calabri, Blanc-Deserti:2015, Blanc-Furter:2013}.

\begin{eg}\label{eg:jonnotcons}
Consider the birational map $f_a(x,y)=(x+a, y\frac{x}{x+1})$, where $a\in \bfk$ is a parameter. In homogeneous
coordinates, $f_a[x:y:z]=[(x+az)(x+z):xy:(x+z)z]$, so  $\deg(f_a)=2$  
for every parameter $a$. The $n$-th iterate of $f_a$ is 
\begin{equation}
f_a^n(x,y)=\left(x+na, \; y\frac{x}{x+1}\frac{x+a}{x+a+1}\cdots \frac{x+(n-1)a}{x+(n-1)a+1}\right).
\end{equation}
It follows that $f_a$ is a Jonqui\`eres twist except if some multiple of $a$ is equal to~$1$. Thus, if the 
characteristic of $\bfk$ is $0$, the set of parameters $a$ such that $f_a$ is elliptic is a countable 
Zariski dense subset of $\bfk$. 
\end{eg}

\begin{eg}\label{eg:jonnotconscarp}
Set $f_\alpha(x,y)=(\alpha x, q(x)y)$ where $q\in \bfk(x)$ and $\alpha\in \bfk^*$. 
If $q(x)=(x-1)/(x-\beta)$, then $f_\alpha$ is a Jonqui\`eres twist if and only if $\beta\notin \alpha^\Z$.
\end{eg}

\begin{thm}\label{thm:example-Jonq}
Let $f_{\alpha}$ be the birational transformation of the plane defined by 
$
f_\alpha(x,y)=(\alpha x, q(x)y)
$
where $\alpha\in \bfk^*$ is not a root of unity  and $q\in \bfk(x)$ is not constant. Assume that the 
set $\alpha^\Q\setminus \{1\}$ does not contain any zero, pole or ratio $z/z'$ of zeroes and/or poles of $q$. 
Then, for every integer
$m>1$ we have: 
\begin{enumerate}
\item $f_\alpha$ is a Jonqui\`eres twist; 
\item $f_\alpha$ is conjugate to $g_{\alpha, m}=(\alpha x, \alpha^m (\alpha x - 1)(x-\alpha^m)^{-1} q(x)y)$;
\item there is no conjugacy between $f_\alpha$ and $g_{\alpha, m}$ of degree $< m$.
\end{enumerate}
\end{thm}

Note that the degrees of $f_\alpha$ and $g_{\alpha,m}$ are bounded by $\deg(q)+1$; so Theorem~\ref{thm:Halphen-Conjugacy-Classes}
has no analogue for Jonqui\`eres twists.

Let us prove Theorem~\ref{thm:example-Jonq}. The $n$-th iterate of $f_\alpha$ is the birational transformation 
$f_\alpha(x,y)=(\alpha^n x, q_n(x) y)$ with
\begin{equation}
q_n(x)=\prod_{j=0}^{n-1}q(\alpha^j x).
\end{equation}
The degree of $q_n$ is $n\deg(q)$ except if there is a zero $z_1$ of $q$, a pole $z_2$ of $q$, and integers
$0\leq j, k\leq n$ such that $(\alpha^jx-z_1)(\alpha^k x-z_2)^{-1}$ is a constant. This happens if, and only
if $z_1/\alpha^j$ is equal to $z_2/\alpha^k$ or, equivalently, if $z_1/z_2$ is in $\alpha^{\Z}$, which is excluded by
hypothesis. Thus, $f_\alpha$ is a Jonqui\`eres twist.

\begin{rem}
We shall work in $\bbP^1_\bfk\times \bbP^1_\bfk$ and use affine coordinates $(x,y)$. 
The indeterminacy points of $f_\alpha$ are the points $(x,y)$ such that $q(x)=0$ and $y=\infty$ or $q(x)=\infty$
and $y=0$; they are located at the intersection of the horizontal lines $y=0$ or $\infty$ and vertical lines
$x=z$ where $z$ is a zero or a pole of $q$.

Let $z$ be a zero of $q$. Then, $f$ contracts the vertical line $\{x=z\}$ onto the point $(\alpha z, 0)$.
The forward orbit of this point is the sequence $f^n(\alpha z, 0)=(\alpha^{n+1}z, 0)$: since $\alpha^n z\neq z'$
for all zeros and poles of $f$, the positive iterates of $f$ are well defined along the orbit of $(\alpha z, 0)$. 
This remark holds also for poles of $q$ and negative orbits of indeterminacy points.
In particular, $f$ is algebraically stable, as a birational map of $\bbP^1\times \bbP^1$.  
\end{rem}

To conjugate $f$ to $g_{\alpha,m}$, consider a birational map 
$h_m(x,y)=(x,s_m(x)y)$, where $s_m$ will be specified soon. Then, 
\begin{equation}
h_m\circ f\circ h_m^{-1}(x,y)=\left(\alpha x, q(x)\frac{s_m(\alpha x)}{s_m(x)} y\right)
\end{equation}
and to obtain $g_{\alpha,m}$ it suffices to choose 
$s_m(x)=(x-1)(x-\alpha)\cdots (x-\alpha^m)$.
For the third assertion of the theorem, we start with the following lemma. 

\begin{lem} 
Let $u$ be a birational transformation of the plane that commutes to $f_\alpha$. There exists 
$\ell\in \Z$ and $\beta\in \bfk^*$ such that $u\circ f_\alpha^\ell(x,y)=(x,\beta y)$.
In other words, the centralizer of $f_\alpha$ is the direct product of $(f_\alpha)^\Z$ and the
multiplicative group, acting on the second factor by $(x,y)\mapsto (x,\beta y)$ for $\beta\in \bfk^*$. 
\end{lem}

\begin{proof} Let $u$ be an element of the centralizer of $f_\alpha$. Since
$f_\alpha$ preserves a unique pencil of curves, $u$ is an element of the Jonqui\`eres group, and since its action on 
the $x$-variable must commute to $x\mapsto \alpha x$, we deduce that 
$u(x,y)=(\epsilon x, B_x(y))$ for some $\epsilon\in \bfk^*$ and $B_x\in \PGL_2(\bfk(x))$. Moreover, the sections $y=0$ and 
$y=\infty$ of the fibration are $f_\alpha$-invariant, and every invariant section coincides with one of these two 
sections; indeed, if $S(x)=(x,s(x))$ is an $f_\alpha$-invariant section, then $s(\alpha^n x)=q_n(x)s(x)$, and if
$s$ is not one of the constants $0$ or $\infty$ this contradicts $\deg(q_n)=n\deg(q)$. Thus, $u$ fixes or permutes those two sections, and we obtain 
\begin{equation}
u(x,y)=(\epsilon x, r(x) y^{\pm 1})
\end{equation}
for some rational function $r$. 
Since $u$ commutes to $f_\alpha^n$, we obtain 
\begin{equation}
r(\alpha^n x) q_n(x)=r(x)q_n(\epsilon x)\quad {\text{or}} \quad r(\alpha^n x) q_n(x)^{-1}=r(x) q_n(\epsilon x).
\end{equation}
In both cases, we see that $q_n(x)$ and $q_n(\epsilon x)$ share many zeros and/or poles; in other words, 
we get equalities of type $z/\alpha^m=z'/(\epsilon \alpha^\ell)$ where $z$ and $z'$ are pairs of zeros and/or poles. 
Iterating $u$, i.e. changing $\epsilon$ into $\epsilon^k$ we can choose $z=z'$ and we see that 
there are integers $k\geq 1$ and $\ell \in \Z$ such that 
$\epsilon^k=\alpha^\ell$ (with $k\leq 2\deg(q)$).  Then, we get a relation $z/z'\in \alpha^\Q$, unless $\epsilon$
is an integral power of $\alpha$. Thus, changing $u$ into $u\circ f_\alpha^\ell$ for some $\ell\in \Z$, we may assume
that $\epsilon=1$. Now, writing $u\circ f_\alpha^\ell(x,y)=(x,r'(x)u)$, the commutation reads 
\begin{equation}
r'(\alpha^n x)  =r'(x) \quad {\text{or}} \quad r'(\alpha^n x) =r'(x) q_n( x)^2.
\end{equation}
In the left-hand case, $r'$ must be a constant because $\alpha$ is not a root of unity. And the right-hand case is excluded because
the degree of $q_n$ is not bounded. 
\end{proof}

We can now conclude the proof of the theorem. Let $h$ be a conjugacy between $f_\alpha$ and $g_{\alpha,m}$, 
then $h\circ h_m^{-1}$ is an element of the centralizer of $f_\alpha$. Thus, there is an integer $n$ and an element
$\beta$ of $\bfk^*$ such that 
$
h=f_\alpha^n\circ \beta\circ h_m=(x,y)\mapsto (\alpha^n x,q_n(x)\beta s_m(x)y).
$
The degree of $q_n s_m$ is equal to the sum $\deg(q_n)+\deg(s_m)$ because the poles of $q$ are distinct from the
zeroes of $s_m$ (otherwise $z/\alpha^k=\alpha^j$ for some pair of integers). Thus, the degree of $h$ is larger than $m$.

%
%

\subsection{Invariant pencils for elliptic}\label{par:elliptic_pencils}

%
%
\begin{thm}\label{thm:bounds for elliptic} 
There is an integer ${\mathrm{ell}}(d)$ such that every elliptic element of $\Bir_d(\bbP^2_\bfk)$ preserves a pencil of degree 
at most ${\mathrm{ell}}(d)$.
\end{thm}

\begin{proof}
If $f$ is an elliptic element of $\Bir(\bbP^2_\bfk)$, denote by ${\mathrm{minv}}(f)$ the minimum degree of an $f$-invariant pencil. 
By contradiction, consider a sequence of elliptic elements $f_m\in \Bir_d(\bbP^2_\bfk)$ for which ${\mathrm{minv}}(f_m)$ increases to $+\infty$ with $m\in \N$.
Let $W$ be the Zariski closure of $\{ f_m\; ; \;m\in \N\}$ in $\Bir_d(\bbP^2_\bfk)$; extracting a subsequence if necessary, we assume that $W$ is irreducible. 
Let $K=\bfk(W)$ be the function field of $W$;
the elements of $W$ define an element $f$ of $\Bir_d(\bbP^2_K)$. If $f$ is a Halphen or Jonquières twist it preserves a pencil of some degree $d_W$, and the proofs of Theorems~\ref{thm:bounds for halphen} and~\ref{thm:bounds for jonquieres} show that 
on a dense open subset of $W$, every element preserves a pencil of  degree $\leq d_W$, so we get a contradiction. It follows that $f$ must be elliptic; as such it preserves some pencil say of degree $\delta$. But then, on a dense open subset of $W$, this pencil specializes into an invariant pencil of degree $\delta$ and, again, we get a contradiction.
\end{proof}
 
%
%

\subsection{Limits}\label{par:limits}

%
%

Working over $\C$, we endow the Cremona group $\Bir(\bbP^2_\C)$ with the euclidean topology introduced in \cite{Blanc-Furter:2013}. A sequence of birational transformations $f_n\colon\bbP^2_\C\dasharrow \bbP^2_\C$ of degree at most $D$ converges towards a birational 
transformation $h$ if and only if  the following equivalent properties are satisfied:

\smallskip 

(A) One can find homogeneous polynomials $(P,Q,R)$ and $(P_n,Q_n, R_n)$ of the same degree $d\leq D$ such that 
$h=[P:Q:R]$, $f_n=[P_n:Q_n:R_n]$, and $(P_n,Q_n,R_n)$ converges towards $(P,Q,R)$  in the vector space of 
triples of homogeneous polynomials (note that $(P,Q,R)$, or the $(P_n,Q_n,R_n)$ may have common factors);

(B) 
The graph of $f_n$ in $\bbP^2(\C)\times \bbP^2(\C)$ converges towards the union of the graph of $h$ and a residual 
algebraic subset whose projection on each factor has codimension $\geq 1$;

(C) There is a non-empty open subset ${\mathcal{U}}$ of $\bbP^2(\C)$ (in the euclidean topology) on which  $h$ and the $f_n$
are regular and $f_n$ converges to  $h$ uniformly. 

\smallskip

Our goal is to understand the possible changes of types that occur when taking limits; for instance, can we obtain any elliptic
element $h$ as a limit of loxodromic (resp. parabolic) elements $f_n$ ? 
More generally, $\C$ will be replaced by any local field of $\bfK$ of characteristic zero, as in \cite{Blanc-Furter:2013}.

\begin{eg} 
By analogy with the Picard-Manin space, consider a finite dimensional hyperbolic space $\Hyp_m\subset \R^{1+m}$ and its group of  isometries 
$\PSO_{1,m}(\R)$. There are three types of isometries, elliptic, 
parabolic, and loxodromic ones, and in the Lie group $\PSO_{1,m}(\R)$ 
\begin{enumerate}
\item every parabolic element is a limit of loxodromic elements; if $m$ is odd, every elliptic element is a limit
of loxodromic element;
\item every parabolic element is a limit of elliptic elements; 
\item there are elliptic elements which are not limits of parabolic elements (resp. of loxodromic elements if $m$ is odd). 
\end{enumerate}
For sequences of pairwise conjugate isometries, the only possible change of type is 
a sequence of parabolic elements converging to an elliptic one; elliptic and loxodromic isometries have closed 
conjugacy classes, but parabolics don't. 
\end{eg}

 

\subsubsection{Limits and types}\label{par:limits_types}
 
Let $(f_n)$ be a sequence of elements of $\Bir(\bbP^2_\C)$
that converges towards $f\in \Bir(\bbP^2_\C)$. Theorem~\ref{thm:Xie} and 
Section~\ref{par:DegBoundHalphen} imply:
(1) {\sl{if  $\,\lambda(f_n)=1$ for all $n$, then $\lambda(f)=1$}}; 
(2) {\sl{if all $f_n$ are Jonqui\`eres twists or elliptic transformations, then $f$ is not a Halphen twist.}}
\begin{eg}
Consider the birational transformation $f_b(x,y)=(x+1,  \frac{xy}{x+b})$. It is elliptic if and only if $b\in \Z$; 
moreover, $f_b$ is conjugate to $f_{b'}$ for any pair of integers $(b,b')\in \Z^2$ (take a conjugacy $h(x,y)=(x+b-b', ys(x))$ 
with $s(x)=x(x+1)\cdots (x+b-b'-1)$). If $\bfK=\Q_p$, then we can take a sequence of integers $a_n\in \Z$ that converges
towards a limit $a_\infty\in \Z_p\setminus \Z$; this gives a sequence of pairwise conjugate 
elliptic elements that converges towards a Jonqui\`eres twist, a phenomenon which is not possible in the group of isometries of 
$\Hyp_m$. Starting with the family $f_\alpha$ from Theorem~\ref{thm:example-Jonq}, but with $\alpha_n$ a sequence of well chosen 
roots of unity, one observes a similar phenomenon in $\Bir(\bbP^2_\C)$.
\end{eg}

\subsubsection{Limits of Halphen twists} 

\begin{thm}
There is a sequence of positive integers ${\mathrm{hal}}(d)$ with the following property. 
Let $h$ be a birational transformation of the plane $\bbP^2_\bfK$ which  is a limit of Halphen twists of degree $\leq d$. 
Then $h$ preserves a pencil of curves of degree $\leq {\mathrm{hal}}(d)$, and the action of $h$ on the base of the pencil has order
$\leq 60$. 
\end{thm} 
 
\begin{proof} The field $\bfK$ has characteristic zero. So, if $X$ is a Halphen surface with a Halphen twist, the image of $\tau\colon \Aut(X)\to \Aut(\bbP^1_{\overline{\bfK}})$ contains at most $60$ elements; this follows from the fact that the invariant fibration $\pi\colon X\to \bbP^1_{\overline{\bfK}}$ 
has at least $3$ and at most $12$ singular fibers (see~\cite[Pro. 7.10]{Grivaux} and \cite[Pro. B]{Gizatullin:1980}). 
Now, assume that the twists $f_n$ converge toward $h$, and apply Theorem~\ref{thm:bounds for halphen}:
along a subsequence, the $f_n$ all have degree $d_0\leq d$,
their invariant pencils $\HP_n$ have degree $D_0\leq {\mathrm{Hal}}(d_0)$, and $f_n^k$ preserves every member of $\HP_n$ 
for some positive integer $k\leq 60$. In other words, the linear system $\HP_n$ is the level set of a rational
function ${\overline{\pi}}_n\colon \bbP^2_\C\dasharrow \bbP^1_\C$ of degree $D_0$ and ${\overline{\pi}}_n\circ f_n^k={\overline{\pi}}_n$.
Each  linear system $\HP_n$ corresponds to a line in the projective space 
$ \bbP(H^0(\bbP^2, {\mathcal{O}}_{\bbP^2}(D_0)))$; thus,
taking a further subsequence the $\HP_n$ converge towards a linear system $\HP_\infty$. The line
$\HP_\infty$ determines  a pencil of plane curves of degree $D_0$, but this linear system may have fixed components; 
removing them, we get a pencil $\HP_\infty'$ of degree $\leq D_0$.
One can find a non-empty open subset $U\subset \bbP^2(\C)$ in the euclidean topology such that on $U$, and along a subsequence, the pencils $\HP_n$ converge towards $\HP'_\infty$ and $f_n^k$ converges towards $h^k$.
Since $f_n^k$ preserves every curve of the pencil $\HP_n$, we deduce that
 $h^k$ preserves every member of the pencil $\HP_\infty'$.  
 \end{proof}
 
\begin{eg} The construction by renormalization described in the next subsection provides examples of Halphen twists converging to elliptic transformations (for instance, every element of $\PGL_3(\C)$ preserving every member of a pencil of plane curves is a limit of pairwise conjugate Halphen twists).  \end{eg}
 
\begin{eg} Consider a pencil of cubic curves with nine distinct base points $p_i$ in $\bbP^2_\bfk$. 
Given a point $m$ in $\bbP^2_\bfk$, draw the line $(p_1m)$ and denote by $m'$ the third intersection point 
of this line with the cubic of our pencil that contains $m$: 
the map $m\mapsto \sigma_1(m)=m'$ is a birational involution. 
Replacing $p_1$ by $p_2$, we get a second involution and, for a very general pencil, $\sigma_1\circ \sigma_2$ 
is a Halphen twist that preserves our cubic pencil. 
At the opposite range, consider the degenerate cubic pencil, the members of which are the union of a line through the origin and the circle $C=\{x^2+y^2=z^2\}$. Choose $p_1=[1:0:1]$ and $p_2=[0:1:1]$
as our distinguished base points. Then, $\sigma_1\circ \sigma_2$ is a Jonqui\`eres twist preserving the pencil of lines
through the origin; if the plane is parameterized by $(s, t)\mapsto (st, t)$, this Jonqui\`eres twist is conjugate to 
$(s,t)\mapsto \left( s, \frac{(s-1)t+1}{(s^2+1)t+s-1}\right)$. 
Now, if we consider a family of general cubic pencils converging towards this degenerate pencil, we obtain a sequence of Halphen twists
converging to a Jonqui\`eres twist. As noted above, the opposite change of type --from Jonqui\`eres to Halphen--  is not possible. 
 \end{eg}

\subsubsection{Renormalization} 
This section adapts an argument already used (at least) by Blanc, Furter, and Maubach (see \cite{Blanc:ENS, Blanc:ANT, Furter-Maubach}). Consider an element $f$ of $\Bir_d(\bbP^2_\bfK)$ with a non-degenerate fixed point $q$: this means that $f$ and $f^{-1}$ are well defined
at $q$ and $f(q)=q$. After a linear change of coordinates, we may assume that $q$ is the origin of the affine plane $\bbA^2_\bfK\subset \bbP^2_\bfK$.
Conjugating $f$ by 
$m(x,y)=(tx,ty)$ with $t\in \bfK$, and
letting $t$ go to $0$, we obtain a family of conjugates of $f$ that converges towards
the linear transformation $Df_{(0,0)}$ in $\Bir_d(\bbP^2_\bfK)$.  

Every element of $\PGL_3(\bfK)$ is conjugate in $\Bir(\bbP^2_{\overline{\bfK}})$ to one of the following transformations of the affine plane
(see~\cite{Blanc:Manuscripta2006, Blanc-Deserti:2015}): 
\begin{itemize}
\item[(a)] $g(x,y)=(ax,by)$ for some non-zero complex numbers $a$ and $b$ in $\overline{\bfK}$;
\item[(b)] $h(x,y)=(bx+b, by)$ for some non-zero complex number $b\in \overline{\bfK}$.
\footnote{Note that $h(x,y)=(x+1, y/b)$ is conjugate in the Cremona group to $h'(x,y)=(b(x+y), by)$.}
\end{itemize}
Both are in the closure of the conjugacy class of a loxodromic element. 
Indeed, in  case (a), consider the following automorphism of the plane: 
\begin{equation}
f(x,y)=(ax+(y+x^2)^2, by + bx^2).
\end{equation}
It is loxodromic, it  fixes the origin, and $Df_{(0,0)}$ coincides with $g(x,y)=(ax,by)$.
In case (b), consider $f'(x,y)=( b^{-1} (x+y + (y+(x+y)^2)^2), b^{-1}(y+(x+y)^2))$.

\begin{thm}
Let $g$ be any element of $\PGL_3(\bfK)$, or any   elliptic element of $\Bir(\bbP^2_\bfK)$ of infinite order. 
Then, $g$ is a limit of pairwise conjugate loxodromic elements $($resp. Jonqui\`eres twists$)$ in the 
Cremona group $\Bir(\bbP^2_{\bfK'})$, for some finite extension $\bfK'$ of $\bfK$.
\end{thm}

Unfortunately, this  does not say anything for finite order elements of $\Bir(\bbP^2_\C)$ 
which are not conjugate to  automorphisms of $\bbP^2_\C$, such as Bertini involutions. 

\begin{proof}
If $g$ is  elliptic and  of infinite order, its conjugacy class intersects $\PGL_3(\overline{\bfK})$, so we assume
$g\in \PGL_3(\bfK')$ for some finite extension of $\bfK$. 
We already explained how the renormalization argument shows that $g$ is a limit
of pairwise conjugate loxodromic elements. To show that $g$ is a limit of conjugate Jonqui\`eres twists, we
need to construct, for  every $(a,b)\in (\bfK^\times)^2$, a Jonqui\`eres twist $f$ fixing 
a point $q$, whose  linear part $Df_q\in \GL_2(\C)$ is conjugate to 
\begin{equation}
\left(\begin{array}{cc} a & 0 \\ 0 & b \end{array}\right) \; \text{ or } \; \left(\begin{array}{cc} b & b \\ 0 & b\end{array}\right).
\end{equation}
The birational transformations $f(x,y)=(a(1+y) x, by)$ and $f'(x,y)=(b(1+y)(x+y), by)$ satisfy these properties 
at $q=(0,0)$.
\end{proof}


%
%

\bibliographystyle{plain}

\bibliography{references_arxiv}

\def\cprime{$'$}
\begin{thebibliography}{10}

\bibitem{Bell2006}
J.~P. Bell.
\newblock A generalised {S}kolem-{M}ahler-{L}ech theorem for affine varieties.
\newblock {\em J. London Math. Soc. (2)}, 73(2):367--379, 2006.

\bibitem{BisiCalabriMella}
C.~Bisi, A.~Calabri, and M.~Mella.
\newblock On plane {C}remona transformations of fixed degree.
\newblock {\em J. Geom. Anal.}, 25(2):1108--1131, 2015.

\bibitem{Blanc:Manuscripta2006}
J.~Blanc.
\newblock Conjugacy classes of affine automorphisms of {$\Bbb K^n$} and linear
  automorphisms of {$\Bbb P^n$} in the {C}remona groups.
\newblock {\em Manuscripta Math.}, 119(2):225--241, 2006.

\bibitem{Blanc:ENS}
J.~Blanc.
\newblock Groupes de {C}remona, connexit\'{e} et simplicit\'{e}.
\newblock {\em Ann. Sci. \'{E}c. Norm. Sup\'{e}r. (4)}, 43(2):357--364, 2010.

\bibitem{Blanc:Algebraic_Elements}
J.~Blanc.
\newblock Algebraic elements of the {C}remona groups.
\newblock In {\em From classical to modern algebraic geometry}, Trends Hist.
  Sci., pages 351--360. Birkh\"{a}user/Springer, Cham, 2016.

\bibitem{Blanc:ANT}
J.~Blanc.
\newblock Conjugacy classes of special automorphisms of the affine spaces.
\newblock {\em Algebra Number Theory}, 10(5):939--967, 2016.

\bibitem{Blanc-Calabri}
J.~Blanc and A.~Calabri.
\newblock On degenerations of plane {C}remona transformations.
\newblock {\em Math. Z.}, 282(1-2):223--245, 2016.

\bibitem{Blanc-Cantat}
J.~Blanc and S.~Cantat.
\newblock Dynamical degrees of birational transformations of projective
  surfaces.
\newblock {\em J. Amer. Math. Soc.}, 29(2):415--471, 2016.

\bibitem{Blanc-Deserti:2015}
J.~Blanc and J.~D\'{e}serti.
\newblock Degree growth of birational maps of the plane.
\newblock {\em Ann. Sc. Norm. Super. Pisa Cl. Sci. (5)}, 14(2):507--533, 2015.

\bibitem{Blanc-Furter:2013}
J.~Blanc and J.-P. Furter.
\newblock Topologies and structures of the {C}remona groups.
\newblock {\em Ann. of Math. (2)}, 178(3):1173--1198, 2013.

\bibitem{Blanc:TrGroups}
J\'{e}r\'{e}my Blanc.
\newblock Sous-groupes alg\'{e}briques du groupe de {C}remona.
\newblock {\em Transform. Groups}, 14(2):249--285, 2009.

\bibitem{Boucksom-Favre-Jonsson:Duke}
S.~Boucksom, C.~Favre, and M.~Jonsson.
\newblock Degree growth of meromorphic surface maps.
\newblock {\em Duke Math. J.}, 141(3):519--538, 2008.

\bibitem{Cantat:Annals}
S.~Cantat.
\newblock Sur les groupes de transformations birationnelles des surfaces.
\newblock {\em Ann. of Math. (2)}, 174(1):299--340, 2011.

\bibitem{Cantat:Compositio}
S.~Cantat.
\newblock Morphisms between {C}remona groups, and characterization of rational
  varieties.
\newblock {\em Compos. Math.}, 150(7):1107--1124, 2014.

\bibitem{Cantat:Survey}
S.~Cantat.
\newblock The {C}remona group.
\newblock In {\em Algebraic geometry: {S}alt {L}ake {C}ity 2015}, volume~97 of
  {\em Proc. Sympos. Pure Math.}, pages 101--142. Amer. Math. Soc., Providence,
  RI, 2018.

\bibitem{CDX:LongArxiv}
S.~Cantat, J.~D{\'e}serti, and J.~Xie.
\newblock Three chapters on {C}remona groups (long version with an appendix).
\newblock {\em arXiv}, 2007.13841:1--60, 2020.

\bibitem{Cantat-Dolgachev}
S.~Cantat and I.~Dolgachev.
\newblock Rational surfaces with a large group of automorphisms.
\newblock {\em J. Amer. Math. Soc.}, 25(3):863--905, 2012.

\bibitem{Cantat-Guirardel-Lonjou}
S.~Cantat, V.~Guirardel, and A.~Lonjou.
\newblock Elements generating a proper normal subgroup of the {C}remona group.
\newblock {\em Int. Math. Res. Not. IMRN}, pages 1--26, 2019.

\bibitem{Cerveau-Lins-Neto:1991}
D.~Cerveau and A.~Lins~Neto.
\newblock Holomorphic foliations in {${\bf C}{\rm P}(2)$} having an invariant
  algebraic curve.
\newblock {\em Ann. Inst. Fourier (Grenoble)}, 41(4):883--903, 1991.

\bibitem{Cossec-Dolgachev:book}
F.~R. Cossec and I.~V. Dolgachev.
\newblock {\em Enriques surfaces. {I}}, volume~76 of {\em Progress in
  Mathematics}.
\newblock Birkh\"auser Boston, Inc., Boston, MA, 1989.

\bibitem{NguyenBD:2017}
N.-B. Dang.
\newblock Degrees of iterates of rational maps on normal projective varieties.
\newblock {\em Proc. Lond. Math. Soc. (3)}, 121(5):1268--1310, 2020.

\bibitem{Dang-Favre}
N.-B. Dang and C.~Favre.
\newblock Spectral interpretations of dynamical degrees and applications.
\newblock {\em arXiv}, arXiv:2006.10262:1--57, 2020.

\bibitem{Deserti}
J.~D\'{e}serti.
\newblock Degree growth of polynomial automorphisms and birational maps: some
  examples.
\newblock {\em Eur. J. Math.}, 4(1):200--211, 2018.

\bibitem{Diller-Favre}
J.~Diller and C.~Favre.
\newblock Dynamics of bimeromorphic maps of surfaces.
\newblock {\em Amer. J. Math.}, 123(6):1135--1169, 2001.

\bibitem{Dolgachev}
I.~V. Dolgachev.
\newblock Rational surfaces with a pencil of elliptic curves.
\newblock {\em Izv. Akad. Nauk SSSR Ser. Mat.}, 30:1073--1100, 1966.

\bibitem{Edixhoven-Romagny:Weil}
Bas Edixhoven and Matthieu Romagny.
\newblock Group schemes out of birational group laws, {N}\'{e}ron models.
\newblock In {\em Autour des sch\'{e}mas en groupes. {V}ol. {III}}, volume~47
  of {\em Panor. Synth\`eses}, pages 15--38. Soc. Math. France, Paris, 2015.

\bibitem{Favre:Bourbaki}
Charles Favre.
\newblock Le groupe de {C}remona et ses sous-groupes de type fini.
\newblock Number 332, pages Exp. No. 998, vii, 11--43. 2010.
\newblock S\'{e}minaire Bourbaki. Volume 2008/2009. Expos\'{e}s 997--1011.

\bibitem{Furter-Maubach}
J.-P. Furter and S.~Maubach.
\newblock A characterization of semisimple plane polynomial automorphisms.
\newblock {\em J. Pure Appl. Algebra}, 214(5):574--583, 2010.

\bibitem{Gizatullin:1980}
M.~H. Gizatullin.
\newblock Rational {$G$}-surfaces.
\newblock {\em Izv. Akad. Nauk SSSR Ser. Mat.}, 44(1):110--144, 239, 1980.

\bibitem{Grivaux}
J.~Grivaux.
\newblock Parabolic automorphisms of projective surfaces (after {M}. {H}.
  {G}izatullin).
\newblock {\em Mosc. Math. J.}, 16(2):275--298, 2016.

\bibitem{EGA-IV-I}
A.~Grothendieck.
\newblock \'{E}l\'{e}ments de g\'{e}om\'{e}trie alg\'{e}brique. {IV}. \'{E}tude
  locale des sch\'{e}mas et des morphismes de sch\'{e}mas. {I}.
\newblock {\em Inst. Hautes \'{E}tudes Sci. Publ. Math.}, (20):259, 1964.

\bibitem{Huckleberry-Zaitsev}
A.~Huckleberry and D.~Zaitsev.
\newblock Actions of groups of birationally extendible automorphisms.
\newblock In {\em Geometric complex analysis ({H}ayama, 1995)}, pages 261--285.
  World Sci. Publ., River Edge, NJ, 1996.

\bibitem{Iskovskikh-Shafarevich}
V.~A. Iskovskikh and I.~R. Shafarevich.
\newblock Algebraic surfaces [ {MR}1060325 (91f:14029)].
\newblock In {\em Algebraic geometry, {II}}, volume~35 of {\em Encyclopaedia
  Math. Sci.}, pages 127--262. Springer, Berlin, 1996.

\bibitem{Janusz:ANF}
Gerald~J. Janusz.
\newblock {\em Algebraic number fields}, volume~7 of {\em Graduate Studies in
  Mathematics}.
\newblock American Mathematical Society, Providence, RI, second edition, 1996.

\bibitem{Kraft:regularization}
H.~Kraft.
\newblock Regularization of rational group actions.
\newblock {\em ar{X}iv:1808.08729}, 2018.

\bibitem{Lech1953}
Christer Lech.
\newblock A note on recurring series.
\newblock {\em Ark. Mat.}, 2:417--421, 1953.

\bibitem{Lins-Neto:2002}
A.~Lins~Neto.
\newblock Some examples for the {P}oincar\'{e} and {P}ainlev\'{e} problems.
\newblock {\em Ann. Sci. \'{E}cole Norm. Sup. (4)}, 35(2):231--266, 2002.

\bibitem{Manin:CubicForms}
Yu.~I. Manin.
\newblock {\em Cubic forms}, volume~4 of {\em North-Holland Mathematical
  Library}.
\newblock North-Holland Publishing Co., Amsterdam, second edition, 1986.
\newblock Algebra, geometry, arithmetic, Translated from the Russian by M.
  Hazewinkel.

\bibitem{Marquis:2014}
L.~Marquis.
\newblock Around groups in {H}ilbert geometry.
\newblock In {\em Handbook of {H}ilbert geometry}, volume~22 of {\em IRMA Lect.
  Math. Theor. Phys.}, pages 207--261. Eur. Math. Soc., Z\"{u}rich, 2014.

\bibitem{Pereira:2002}
J.~V. Pereira.
\newblock On the {P}oincar\'{e} problem for foliations of general type.
\newblock {\em Math. Ann.}, 323(2):217--226, 2002.

\bibitem{Platonov-Rapinchuk:Book}
V.~Platonov and A.~Rapinchuk.
\newblock {\em Algebraic groups and number theory}, volume 139 of {\em Pure and
  Applied Mathematics}.
\newblock Academic Press, Inc., Boston, MA, 1994.
\newblock Translated from the 1991 Russian original by Rachel Rowen.

\bibitem{Poonen2014}
B.~Poonen.
\newblock {$p$}-adic interpolation of iterates.
\newblock {\em Bull. Lond. Math. Soc.}, 46(3):525--527, 2014.

\bibitem{Serre:Bourbaki-Cremona}
J.-P. Serre.
\newblock Le groupe de {C}remona et ses sous-groupes finis.
\newblock Number 332, pages Exp. No. 1000, vii, 75--100. 2010.
\newblock S\'{e}minaire Bourbaki. Volume 2008/2009. Expos\'{e}s 997--1011.

\bibitem{Sibony:Panorama}
N.~Sibony.
\newblock Dynamique des applications rationnelles de {$\bold P^k$}.
\newblock In {\em Dynamique et g\'{e}om\'{e}trie complexes ({L}yon, 1997)},
  volume~8 of {\em Panor. Synth\`eses}, pages ix--x, xi--xii, 97--185. Soc.
  Math. France, Paris, 1999.

\bibitem{Stevenhagen-Lenstra}
P.~Stevenhagen and H.~W. Lenstra, Jr.
\newblock Chebotar\"{e}v and his density theorem.
\newblock {\em Math. Intelligencer}, 18(2):26--37, 1996.

\bibitem{TTTruong}
T.~Trung~Truong.
\newblock Relative dynamical degrees of correspondances over fields of
  arbitrary characteristic.
\newblock {\em J. Reine Angew. Math.}, 758:139--182, 2020.

\bibitem{Urech}
C.~Urech.
\newblock Remarks on the degree growth of birational transformations.
\newblock {\em Math. Res. Lett.}, 25(1):291--308, 2018.

\bibitem{Weil:1955}
A.~Weil.
\newblock On algebraic groups of transformations.
\newblock {\em Amer. J. Math.}, 77:355--391, 1955.

\bibitem{Xie:Duke}
J.~Xie.
\newblock Periodic points of birational transformations on projective surfaces.
\newblock {\em Duke Math. J.}, 164(5):903--932, 2015.

\bibitem{Xie2019}
J.~Xie.
\newblock The existence of {Z}ariski dense orbits for endomorphisms of
  projective surfaces (with an appendix in collaboration with {T}homas
  {T}ucker).
\newblock 2019.

\bibitem{Zaitsev:1995}
D.~Zaitsev.
\newblock Regularization of birational group operations in the sense of {W}eil.
\newblock {\em J. Lie Theory}, 5(2):207--224, 1995.

\end{thebibliography}

\end{document}